\documentclass[12pt]{amsart}
\usepackage{amsfonts}
\usepackage{graphicx}
\usepackage{xcolor}
\usepackage{amsmath}
\usepackage{amssymb}
\usepackage{soul}

\usepackage{enumitem}
\setlist[enumerate]{topsep=2mm,parsep=0pt,partopsep=0mm,itemsep=1mm,leftmargin=15mm,labelsep=2mm}

\newtheorem{theorem}{Theorem}
\newtheorem{proposition}{Proposition}[section]
\newtheorem{corollary}[proposition]{Corollary}
\newtheorem{lemma}[proposition]{Lemma}
\newtheorem{result}{Theorem}
\theoremstyle{remark}
\newtheorem{remark}[proposition]{Remark}

\newtheorem{definition}[proposition]{Definition}

\numberwithin{equation}{section}

\newcommand{\proofof}[1]{\fontseries{bx}\fontshape{it}\selectfont Proof of #1}

\newcommand{\ind}{\boldsymbol{1}}
\let\Re\undefined
\let\Im\undefined
\DeclareMathOperator{\Re}{\mathtt{Re\hskip-.07ex}}
\DeclareMathOperator{\Im}{\mathtt{Im}}

\let\ge=\geqslant
\let\le=\leqslant

\newcommand{\C}{{\mathbb{C}}}
\newcommand{\R}{{\mathbb{R}}}
\newcommand{\N}{{\mathbb{N}}}
\let\Complex=\C
\let\Real=\R
\let\Natural=\N

\newcommand{\UH}{{\mathbb{H}}}
\newcommand{\U}{\mathfrak{U}\hskip.05em}
\newcommand{\UD}{{\mathbb{D}}}
\newcommand{\UC}{{\partial\UD}}
\newcommand{\di}{\mathrm{d}}
\newcommand{\DI}{\,\di}

\newcommand{\Hol}{{\sf Hol}}

\def\id{\mathop{\mathrm{id}}}
\newcommand{\anglim}{\angle\lim}
\newcommand{\neqC}{\subsetneq\C}
\newcommand{\DC}{{D\neqC}}

\newcommand{\BF}{\mathfrak{BF}}
\newcommand{\Gen}{\mathcal{G}}

\newcommand{\mcite}[1]{\csname b@#1\endcsname}

\newcommand{\REM}[1]{\relax}

\begin{document}
\title[Loewner Theory for Bernstein functions I]{Loewner Theory for Bernstein functions I: evolution families and differential equations}
\author[P.~Gumenyuk, T.~Hasebe, J.\,L. P\'erez]{Pavel Gumenyuk, Takahiro Hasebe, Jos\'e-Luis P\'erez}

\address{P.~Gumenyuk: Department of Mathematics, Politecnico di Milano, via E.~Bonardi~9, 20133 Milan, ITALY} \email{pavel.gumenyuk@polimi.it}

\address{T.~Hasebe: Department of Mathematics, Hokkaido University, North~10, West~8, Kita-ku, Sapporo 060-0810, JAPAN}
\email{thasebe@math.sci.hokudai.ac.jp}
\thanks{This work is supported by the Research Institute for Mathematical Sciences,
an International Joint Usage/Research Center located in Kyoto University, and by JSPS Grant-in-Aid for Young Scientists 19K14546, JSPS Scientific Research 18H01115 and JSPS Open Partnership Joint Research Projects grant no.~JPJSBP120209921.}

\address{J.\,L. P\'erez: Department of Probability and Statistics, Centro de Investigaci\'on en Matem\'aticas, A.C., Calle Jalisco s/n, C.P. 36240, Guanajuato, MEXICO}
\email{jluis.garmendia@cimat.mx}


\begin{abstract}
One-parameter semigroups of holomorphic functions appear naturally in various applications of Complex Analysis, and in particular, in the theory of (temporally) homogeneous Markov processes. A suitable analogue of one-parameter semigroups in the inhomogeneous setting is the notion of a (reverse) evolution family. In this paper we study evolution families formed by Bernstein functions, which play the role of Laplace exponents for inhomogeneous continuous-state branching processes. In particular, we characterize all Herglotz vector fields that generate such evolution families and give a complex-analytic proof of a qualitative description equivalent to Silverstein's representation formula for the infinitesimal generators of one-parameter semigroups of Bernstein functions. We also establish several sufficient conditions for families of holomorphic self-maps, satisfying the algebraic part in the definition of an evolution family, to be absolutely continuous and hence to be described as solutions to the generalized Loewner\,--\,Kufarev differential equation. Most of these results are then applied in the sequel paper~\cite{Paper2} to study continuous-state branching processes.
\vskip-2cm
\end{abstract}

\maketitle

\tableofcontents%\newpage
\vskip2cm

\section{Introduction}
Particular interest to look into families of Bernstein functions from the viewpoint of Loewner Theory comes from Probability Theory.
In general, complex-analytic tools are known to play an important role in the study of Markov branching processes, see e.g. \cite{BranchingProcesses, Goryainov-survey, Harris}. A very classical example is the Galton\,--\,Watson process, which is a (temporally) homogeneous Markov chain with the one-step transition probabilities encoded in the so-called generating (holomorphic) function $f\colon \UD\to\UD:=\{{z\in\Complex\colon}|z|<1\}$ as its Taylor coefficients at~$z=0$. Evolution of the stochastic process corresponds to a (deterministic) dynamics in a complex domain: the generating function for the $n$-step transition probabilities coincides with the $n$-th iterate of~$f$ defined as usual by
$$
   f^{\circ n}:=\underbrace{\,f\circ\ldots\circ f}_{n\text{~times}}\colon \UD\to\UD.
$$
In case of a homogeneous branching process with continuous time, iterates are replaced by a continuous one-parameter semigroup of holomorphic self-maps, see Section~\ref{SS_one-param} for the definition and basic properties.

In a similar manner, inhomogeneous branching processes lead to non-autonomous holomorphic dynamical systems. In this setting the role of one-parameter semigroups is played by the so-called (reverse) evolution families, see Section~\ref{SS_REF}.

Independently from applications to stochastic processes,  one-parameter semigroups and iteration of holomorphic functions have been a subject of deep studies in Complex Analysis since early 1900s. Evolution families, in turn, constitute  one of the cornerstone concepts in Loewner Theory, which originated from Loewner's seminal paper~\cite{Loewner} of~1923 and since then has been serving as a powerful tool in the theory of conformal mappings, see e.g. \cite{AbateBracci:2010,Loewner-survey} for a detailed survey.

Applications of Loewner Theory to branching processes (with discrete states) were studied in detail by Goryainov~[\mcite{Goryainov}\,--\,\mcite{Goryainov3}]. It is worth mentioning that similar relationships between stochastic processes and evolution families (or one-parameter semigroups in the homogeneous case) have been recently discovered in Non-Commutative Probability, see e.g. \cite{Bauer,UFranz,Sebastian,FHS-Diss}.

This paper is motivated by yet another interesting case, in which the transition probabilities can be encoded in an evolution family formed by holomorphic self-maps of the right half-plane. Continuous-state (Markov) branching processes introduced in 1958 by Ji\v{r}ina~\cite{MJ1958}, see also~\cite{Gre74}, have been attracting increasing interest in the last decades, see e.g.~\cite{FangLi} or~\cite{Li2020} and references therein. Because of the branching property, the transition probability measures $k_{s,t}(x,\cdot)$ in such a process~$(X_t)$ are infinitely divisible and, as a result,  Bernstein functions come into play via the Laplace transform:
$$
\mathbb E\big[e^{-\zeta X_t}\big|X_s=x\big]~=\,\int_0^{+\infty}e^{-\zeta y}\, k_{s,t}(x,\di y)~=~\exp(-xv_{s,t}(\zeta)),\quad {x,\,\zeta\ge0},
$$
where $v_{s,t}\colon [0,+\infty)\to[0,+\infty)$ is a Bernstein function, which is uniquely defined by the transition kernel~$k_{s,t}$ and which we refer to as \textsl{the Laplace exponent} of the process.
In turn, the Chapman\,--\,Kolmogorov equation  leads to a sort of semigroup property $v_{s,t}\circ v_{t,u}=v_{s,u}$ for any $s,t,u$ with $0\le s\le t\le u$; essentially, this means that the Laplace exponents form a reverse evolution family. For more details we refer the reader to our sequel paper~\cite{Paper2}.

While homogeneous continuous-state branching processes are well understood, see e.g. \cite[Chapter~10]{Kyp14}, a systematic study of the inhomogeneous case has been launched only recently. In particular, in~\cite{FangLi} the authors construct inhomogeneous continuous-state branching processes starting from an integral evolution equation for the Laplace exponents~$v_{s,t}$ regarded as self-maps of~$(0,+\infty)$.

A well-known remarkable fact is that every Bernstein function $f\not\equiv0$ extends to a holomorphic self-map of the right half-plane  $\UH:={\{z\in\C\colon\Re z>0\}}$. Another important property is that the composition of two Bernstein functions is again a Bernstein function. In other words, the holomorphic extensions of Bernstein functions $f\not\equiv0$ form a subsemigroup $\BF$ in the topological semigroup $\Hol(\UH,\UH)$ of all holomorphic self-maps of~$\UH$ endowed with the operation of composition and the usual topology of locally uniform convergence. Moreover, $\BF$ is topologically closed in~$\Hol(\UH,\UH)$. See Section~\ref{SS_Bernstein} for more details.

The purpose of this paper is to develop complex-analytic methods based on the modern Loewner Theory which would be useful, in particular, for studying inhomogeneous continuous-state branching processes. A part of our results are then applied in the sequel paper~\cite{Paper2}. Other results are aimed to future investigation of branching processes and/or represent some general interest in the frames of Loewner Theory and its applications.

\subsection{Organization of the paper and summary of main results}
The paper is organized as follows. In the next section, Section~\ref{S_preliminaries}, we present necessary background on one-parameter semigroups and (reverse) evolution families and on boundary fixed points of holomorphic self-maps; the last subsection is devoted to basic properties of Bernstein functions (regarded as holomorphic functions in the right-half plane~$\UH$).

In Section~\ref{S_results} we establish our main results.
We start by proving Theorems~\ref{TH_generators-convex-cone} and~\ref{TH_EvolFam-HVF} in Sections~\ref{SS_one-param-in-semigroups} and~\ref{SS_EF-in-semigroups}, respectively, which concern infinitesimal description of an arbitrary topologically closed (sub)semigroup $\U$ in $\Hol(D,D)$, where $\DC$ is a simply connected domain.
Theorem~\ref{TH_generators-convex-cone} is a general version of a result known for several special cases; it states that the infinitesimal generators associated to one-parameter semigroups in~$\U$ form a topologically closed convex cone. Theorem~\ref{TH_EvolFam-HVF} deals with the evolution families. We prove that an evolution family $(w_{s,t})$ generated by a Herglotz vector field~$\phi$ is contained in~$\U$ if and only if for a.e.\!{} ``frozen'' instant~$t$, $\phi(\cdot,t)$ is an infinitesimal generator of a one-parameter semigroup in~$\U$.
Both theorems apply to the case of the semigroup $\BF\subset\Hol(\UH,\UH)$ consisting of all Bernstein functions different from the identical zero.

Theorem~\ref{TH_EvolFam-HVF} reduces the study Herglotz vector fields generating evolution families of Bernstein functions to the study of infinitesimal generators associated to one-parameter semigroups in~$\BF$. Such generators, which we refer to as \textsl{Bernstein generators}, appear naturally in Stochastic Processes. A representation formula for Bernstein generators was found long time ago by Silverstein~\cite{Sil68}. Using complex-analytic tools, in particular, the theory of one-parameter semigroups, in Section~\ref{SS_Bernstein-generators} we obtain a qualitative characterization of Bernstein generators, see Theorem~\ref{TH_BG}. As mentioned after the statement of the theorem, it is easily seen to be equivalent to Silverstein's representation. Moreover, as corollaries, we obtain analogous representation of Bernstein generators for the cases when the corresponding one-parameter semigroup has a boundary regular fixed point at $0$ or at~$\infty$.

Section~\ref{SS_AC-conditions} is devoted to finding sufficient conditions for families satisfying the first two (algebraic) conditions in the definitions of reverse and usual evolution families, see Section~\ref{SS_REF}, to be absolutely continuous, i.e. to admit associated Herglotz vector fields and hence to be described with the help of (a general version of) the Loewner\,--\,Kufarev equations \eqref{EQ_ini-LK-ODE}\,--\,\eqref{EQ_ini-invLK-ODE}.
Theorems~\ref{TH_diff-one-trajectory} and~\ref{TH_reverse-diff-one-trajectory} deal with general (usual and reverse) evolution families having a boundary Denjoy\,--\,Wolff point, while Theorem~\ref{TH_AC-BF-EF} and its Corollary~\ref{CR_AC-BF-REF} concern specifically families of Bernstein functions and give sufficient conditions for absolute continuity in terms of the local behaviour at~$0$. Motivated by the latter result, in the last section, Section~\ref{SS_2nd-derovative}, we make sure that for absolutely continuous evolution families in~$\BF$, the well-known formulas for the derivatives of the first and the second order w.r.t. the initial condition at an internal fixed point can be extended to the boundary regular fixed point at~$0$.

\section{Preliminaries}\label{S_preliminaries}
\subsection{One-parameter semigroups of holomorphic self-maps}\label{SS_one-param}
We start by introducing some general notation and terminology. For an arbitrary domain $D\subset\Complex$ we denote by $\Hol(D,\C)$ the set of all holomorphic functions in~$D$. We endow $\Hol(D,\C)$ with the open-compact topology. It is easy to see that a sequence $(f_n)$ converges in ${\Hol(D,\C)}$ if and only if it converges (to the same limit function) locally uniformly in~$D$. Furthermore, for a set $E\subset\C$ we denote by $\Hol(D,E)$ the topological subspace of $\Hol(D,\C)$ formed by all $f\in\Hol(D,\C)$ such that $f(D)\subset E$.
\begin{definition}\label{DF_semigr-of-selfmaps}
By a \textsl{semigroup of holomorphic self-maps} of a domain ${D\subset\C}$ we mean a set $\U\subset\Hol(D,D)$ containing the identity map~$\id_D$ and such that $f\circ g\in\U$ for any~$f,g\in\U$.
We say that such a semigroup $\U$ is \textsl{topologically closed} if $\U$ is (relatively) closed in~$\Hol(D,D)$.
\end{definition}
\begin{definition}
Let $\U$ be a semigroup of holomorphic self-maps of some domain ${D\subset\C}$.
A \textsl{one-parameter semigroup} in~$\U$ is a continuous semigroup homomorphism ${[0,+\infty)\ni t\mapsto v_t\in\U}$ from the semigroup ${\big([0,+\infty),+\big)}$ with the Euclidian topology to the semigroup $(\U,\circ)$ endowed with the topology inherited from~$\Hol(D,D)$.
\end{definition}

Equivalently, $(v_t)\subset\U$ is a one-parameter semigroup if and only if it satisfies the following three conditions: (i)~$v_0=\id_D$; (ii)~$v_{s}\circ v_t=v_{s+t}$ for any $s,t\ge0$; (iii)~$v_t\to\id_D$ in $\Hol(D,D)$ as $t\to0^+$. Thanks to Montel's normality criterion, see e.g. \cite[\S\,II.7, Theorem~1]{Goluzin}, if $\Complex\setminus D$ contains at least two distinct points, then the continuity condition~(iii) is equivalent to the pointwise convergence: $v_t(z)\to z$ as~$t\to0^+$ for each $z\in D$.

\begin{remark}\label{RM_semigr-hyperb-simply-conn}
It is known, see e.g.\,\cite[\S8.4]{BCD-Book}, that in one complex variable, the theory of one-parameter semigroups of holomorphic self-maps is non-trivial only in the following two cases:
\begin{itemize}
\item[(a)] $D$ is conformally equivalent to the unit disk ${\UD:=\{z\colon |z|<1\}}$, i.e. $D$ is simply connected and different from the whole plane~$\C$; \textit{or}
\item[(b)] $D$ is conformally equivalent to $\UD^*:=\UD\setminus\{0\}$.
\end{itemize}
Moreover, case~(b) reduces easily to case~(a) because any one-parameter semigroup~$(\phi_t)$ of holomorphic self-maps of~$\UD^*$ extends holomorphically to~${z=0}$ by setting $\phi_t(0)=0$ for all~$t\ge0$; see \cite[Theorem~8.4.7 on p.\,224]{BCD-Book}.
\end{remark}

Most of the theory in the literature is developed for ${D:=\UD}$ or ${D:=\UH}$, sometimes for the strip ${D:=\{z\colon 0<\Im z<\pi\}}$. For the sake of notational uniformity, we will keep working with an arbitrary simply connected $\DC$, unless explicitly stated otherwise. Clearly, most of the results proved for any particular choice of~$D$ extend, more or less automatically, to the general case with the help of conformal mappings. In particular, it is known \cite{BP78}, see also~\cite{BCM-AMP}, \cite[Chapter~10]{BCD-Book}, or~\cite[\S3.2]{ShoikhetBook} that for any one-parameter semigroup $(v_t)$ in~$\Hol(D,D)$ the limit
\begin{equation}\label{EQ_generator}
 \phi:=\lim_{t\to 0^+}\frac{\id_D-v_t}{t}
\end{equation}
exists in $\Hol(D,\C)$.
 The holomorphic function~$\phi$ is called the \textsl{infinitesimal generator} of~$(v_t)$. The one-parameter semigroup can be recovered from its infinitesimal generator; namely, for each $z\in D$ the map ${[0,+\infty)\ni t\mapsto w:=v_t(z)}$ is the unique solution to
\begin{equation}\label{EQ_autonomous-LK}
\frac{\di w(t)}{\di t}+\phi(w(t))=0,\quad w(0)=z.
\end{equation}
Conversely, if $\phi\in\Hol(D,\C)$ is a semicomplete vector field in~$D$, i.e.\ if for any $z\in D$ the initial value problem~\eqref{EQ_autonomous-LK} has a unique solution $w=w_{z}(t)\in D$ defined for all~$t\ge0$, then the mappings ${v_t(z):=w_z(t)}$, ${z\in D}$, ${t\ge0}$, form a one-parameter semigroup in $\Hol(D,D)$, whose infinitesimal generator coincides with~$\phi$.

In the literature one can find two different conventions concerning the choice of the sign in formula~\eqref{EQ_generator}. Here we follow the convention from~\cite{ShoikhetBook}, which is also often adopted in the literature on Probability Theory. The opposite choice of the sign is made in most of the works in Loewner Theory, e.g. in \cite{BCM1,BRFP2015,CDG_decr}, as well as in many publications on one-parameter semigroups, including~\cite{BCD-Book}.

In what follows, the set of all infinitesimal generators of one-parameter semigroups in the domain~$D$ will be denoted by~$\Gen$, or by~$\Gen\big(\Hol(D,D)\big)$ when we need to specify the domain $D$ explicitly. Furthermore, for a semigroup $\U\subset\Hol(D,D)$, we will denote by $\Gen(\U)$ the set of all $\phi\in\Gen$ for which the associated one-parameter semigroup $(v^\phi_t)$ is contained in~$\U$.

\begin{remark}\label{RM_conformal-trans-of-generators}
Let $\DC$ be a simply connected domain and $f$ be a conformal mapping of $\UD$ onto~$D$. With the help of the variable change $w=f(\omega)$ in equation~\eqref{EQ_autonomous-LK} we see that $\phi:D\to\C$ belongs to $\Gen\big(\Hol(D,D)\big)$ if and only if the function $\psi:=(\phi\circ f)/f'$ belongs to~$\Gen\big(\Hol(\UD,\UD)\big)$.
\end{remark}

\subsection{Absolutely continuous evolution families}\label{SS_REF}
Throughout this section, we fix an interval ${I\subset\R}$ of the form ${I:=[0,T]}$ with ${T\in(0,+\infty)}$ or ${I:=[0,T)}$ with ${T\in(0,+\infty]}$. Furthermore, by $\Delta(I)$ we denote the set $\{(s,t)\colon s,t\in I,\,s\le t\}$. If~${I=[0,+\infty)}$, instead of $\Delta(I)$ we will simply write~$\Delta$.
\newcommand{\AC}{AC_{\mathrm{loc}}}
Finally, $\AC(I)$ will stand for the class of all (real- or complex-valued) functions on~$I$ which are absolutely continuous on each compact subinterval of~$I$.

There are two ways to extend the notion of a one-parameter semigroup to the non-autonomous (or, in probabilistic terminology, time-inhomogeneous) case. The following definitions are slight modifications of the definitions given in \cite{BCM1,CDG_decr}. Most of the results proved in these papers can be easily extended to our setting.
\begin{definition}\label{DF_EF}
An \textsl{(absolutely continuous) evolution family} in a  simply connected domain $\DC$ over the interval~$I$ is a two-parameter family $${(w_{s,t})_{(s,t)\in \Delta(I)}\subset\Hol(D,D)}$$ satisfying the following conditions:
\begin{itemize}
\item[EF1:] $w_{s,s}=\id_{D}$ for any~$s\in I$;

\item[EF2:] $w_{s,u}=w_{t,u}\circ w_{s,t}$ for any $s,t,u\in I$ with ${s\le t \le u}$;

\item[EF3:] for each $z\in D$ there exists a non-negative function $f_{z}\in \AC(I)$ such that
\[
|w_{s,u}(z)-w_{s,t}(z)|\le\int_t^u\!f_{z}(r)\,\di r
\]
for any $s,t,u\in I$ with ${s\le t \le u}$.
\end{itemize}
\end{definition}

\begin{definition}\label{DF_REF}
An \textsl{(absolutely continuous) reverse evolution family} in a simply connected domain $\DC$ over the interval~$I$ is a two-parameter family ${(v_{s,t})_{(s,t)\in \Delta(I)}\subset\Hol(D,D)}$ satisfying the following conditions:
\begin{itemize}
\item[REF1:] $v_{s,s}=\id_{D}$ for any~$s\in I$;

\item[REF2:] $v_{s,u}=v_{s,t}\circ v_{t,u}$ for any $s,t,u\in I$ with ${s\le t \le u}$;

\item[REF3:] for each $z\in D$ there exists a non-negative function $f_{z}\in \AC(I)$ such that
\[
|v_{s,u}(z)-v_{s,t}(z)|\le\int_t^u\!f_{z}(r)\,\di r
\]
for any $s,t,u\in I$ with ${s\le t \le u}$.
\end{itemize}
\end{definition}

If $(w_{s,t})$ satisfies Definition~\ref{DF_EF} with condition EF3 replaced by the requirement that the map
$
 {\Delta(I)\ni(s,t)\mapsto w_{s,t}\in\Hol(D,D)}
$
is continuous, then it is called a \textsl{topological evolution family}. In a similar way, \textsl{topological reverse evolution families} can be defined.

\begin{remark}
Clearly, if $(v_{t})\subset\Hol(D,D)$ is a one-parameter semigroup, then $(w_{s,t})_{0\le s\le t}$ defined by $w_{s,t}:=v_{t-s}$ for all $(s,t)\in\Delta$ is an absolutely continuous evolution family and, at the same time, an absolutely continuous reverse evolution family.
\end{remark}

\newcommand{\Tau}{S}
\begin{remark}\label{RM_REF-EF}
A family ${(v_{s,t})_{(s,t)\in \Delta(I)}\subset\Hol(D,D)}$ is a reverse evolution family if and only if for any $\Tau\in I$ the formula $w_{s,t}:=v_{\Tau-t,\Tau-s}$, $0\le s\le t\le \Tau$, defines an evolution family $(w_{s,t})$ over~${[0,\Tau]}$. The same is true with ``reverse evolution family'' and ``evolution family'' interchanged. For \textit{topological} (reverse) evolution families, these assertions follow directly from the definitions, while in the \textit{absolutely continuous} case, the proof is less trivial, see \cite[Proposition~4.3]{CDG_decr}.
\end{remark}

An observation similar to Remark~\ref{RM_semigr-hyperb-simply-conn} is valid for (reverse) evolution families, see~\cite{CD_top-LTh}, which explains why assuming that $\DC$ is simply connected is of no loss of generality.
At the same time, notwithstanding Remark~\ref{RM_REF-EF}, both Definitions~\ref{DF_EF} and~\ref{DF_REF} seem to be useful in our context. Reverse evolution families, as mentioned in Introduction, spring up in Probability Theory, e.g. in connection with branching processes.  On the other hand, usual evolution families are more natural from the Dynamics point of view. In particular, absolutely continuous evolution families, similarly to one-parameter semigroups, are semiflows of certain holomorphic vector fields in~$D$, but in this case the vector fields are \textit{non-autonomous}, i.e. they may depend explicitly on time. In order to state this result in a precise form, we need the following definition.

\newcommand{\Lloc}{L^1_{\mathrm{loc}}}
By $\Lloc(I)$ we denote the class of all measurable (real- or complex-valued) functions which are integrable on each compact subinterval of~$I$.

\begin{definition}
A function $\phi:D\times I\to\C$ is said to be a \textsl{Herglotz vector field} in a simply connected domain $\DC$ if it satisfies the following three conditions:
\begin{itemize}
\item[HVF1:] for any $t\ge0$, $\phi(\cdot,t)\in\Gen\big(\Hol(D,D)\big)$;
\item[HVF2:] for any $z\in D$, $\phi(z,\cdot)$ is measurable on~$I$;
\item[HVF3:] for any compact set $K\subset D$, there is a non-negative function ${M_K\in\Lloc(I)}$ such that
$$
\max_{z\in K}|\phi(z,t)|\le M_K(t)\quad\text{a.e.~$t\in I$}.
$$
\end{itemize}
\end{definition}
\begin{remark}\label{RM_conformal-trans-of-HVFs}
An observation analogous to Remark~\ref{RM_conformal-trans-of-generators} holds for Herglotz vector fields. Namely, if $f$ maps $\UD$ conformally onto a domain~$D$, then $\phi:D\times I\to\C$ is a Herglotz vector field in~$D$ if and only if $\psi(\zeta,t):={\phi\big(f(\zeta),t\big)/f'(\zeta)}$, ${\zeta\in\UD}$, ${t\in I}$, is a Herglotz vector field in~$\UD$.
\end{remark}
As usual we identify two Herglotz vector fields ${\phi,\psi:D\times I\to\C}$  if $\phi(\cdot,t)=\psi(\cdot,t)$ for a.e. ${t\in I}$. Some of the main results of~\cite{BCM1} can be now summarized and stated in our notation as follows. For $s\in I$, let $I_s:=I\cap[s,+\infty)$.
\begin{result}\label{TH_mainBCM1}
Let $\DC$ be a simply connected domain and $(w_{s,t})$ an absolutely continuous evolution family in~$D$ over the interval~$I$. Then there exists a unique Herglotz vector field ${\phi:D\times I\to\C}$ such that for any $z\in D$ and any $s\in [0,T)$, the map ${I_s\ni t\mapsto w(t):=w_{s,t}(z)\in D}$ is a solution to the initial value problem
\begin{equation}\label{EQ_ini-LK-ODE}
\frac{\di w}{\di t}+\phi(w,t)=0,\quad \text{a.e.\ } t\in I_s;\qquad w(s)=z.
\end{equation}

Conversely, let ${\phi:D\times I\to\C}$ be a Herglotz vector field. Then for any $z\in D$ and any $s\in I$ the initial value problem~\eqref{EQ_ini-LK-ODE} has a unique solution ${t\mapsto w=w(t;z,s)}$ defined on some interval containing the set~$I_s$. Setting ${w_{s,t}(z):=w(t;z,s)}$ for all ${z\in D}$ and all $(s,t)\in\Delta(I)$ one obtains an absolutely continuous evolution family $(w_{s,t})$ over the interval~$I$.
\end{result}

According to Theorem~\ref{TH_mainBCM1}, there is a one-to-one correspondence between absolutely continuous evolution families and Herglotz vector fields. An evolution family and the vector field, corresponding to this evolution family in the sense of Theorem~\ref{TH_mainBCM1}, are usually said to be \textsl{associated} to each other. An analogous one-to-one correspondence exists between absolutely continuous \textit{reverse} evolution families and Herglotz vector fields, see \cite[Theorem~4.2\,(i)]{CDG_decr}. In this respect, if $(v_{s,t})_{(s,t)\in\Delta(I)}$ is an absolutely continuous reverse evolution family and $\phi$ is the Herglotz vector field associated to~$(v_{s,t})$, then for each ${z\in D}$ and each ${t\in I\setminus\{0\}}$ the map ${[0,t]\ni s\mapsto v(s):=v_{s,t}(z)}$ is the unique solution to the initial value problem
\begin{equation}\label{EQ_ini-revLK-ODE}
\frac{\di v}{\di s}=\phi(v,s)\quad \text{for a.e.\ } s \in [0,t];\qquad v(t)=z.
\end{equation}
\begin{remark}
Although the correspondences between a Herglotz vector field~$\phi$ on the one side and the associated absolutely continuous (usual and reverse) evolution families $(w_{s,t})$ and $(v_{s,t})$ on the other side are expressed rather explicitly, there does not seem to exist any simple way to relate the families $(w_{s,t})$ and $(v_{s,t})$ directly, i.e. without involving the Herglotz vector field~$\phi$.
\end{remark}

\begin{remark}
Clearly, if $(v_{s,t})$ is an absolutely continuous reverse evolution family over the interval~$I$, then for each fixed $s\in[0,T)$, the mappings $\tilde v_{s'\!,t'}:=v_{s+s'\!,s+t'}$ also form an absolutely continuous reverse evolution family over a suitable interval. Therefore, combining \cite[Theorems 4.1, 4.2\,(ii) and 1.11]{CDG_decr}, one can easily conclude that there is an equivalent way to express the one-to-one correspondence between Herglotz vector fields and absolutely continuous reverse evolution families introduced above. Namely, as a function of the parameter $t$, any absolutely continuous reverse evolution family~$(v_{s,t})$ satisfies, in the sense of \cite[Definition~2.1]{CDG_decr},  the PDE
\begin{equation}\label{EQ_ini-revLK-PDE}
\frac{\partial v_{s,t}(z)}{\partial t}+\phi(z,t)\frac{\partial v_{s,t}(z)}{\partial z}=0,\quad \text{a.e.\ } t\in I_s,~z\in D;\qquad v_{s,s}=\id\nolimits_D,
\end{equation}
where $\phi$ is the Herglotz vector field associated with~$(v_{s,t})$. Conversely, given any Herglotz vector field $\phi:D\times I\to\Complex$, for each $s\in[0,T)$ the initial value problem~\eqref{EQ_ini-revLK-PDE} has a unique solution $(z,t)\mapsto v_{s,t}(z)$, which is defined for all ${(z,t)\in D\times I_s}$, and moreover, $(v_{s,t})_{(s,t)\in\Delta(I)}$ is exactly the reverse evolution family associated with~$\phi$.
\end{remark}
According to \cite[Theorem~1.11]{CDG_decr}, the inverse mappings $v_{s,t}^{-1}$ also can be recovered by solving an ODE driven by the Herglotz vector field~$\phi$ associated with~$(v_{s,t})$. Namely, for each ${s\in[0,T)}$ and each ${z\in D}$, the map $t\mapsto w(t):=v_{s,t}^{-1}(z)$ is the unique solution to the initial value problem
\begin{equation}\label{EQ_ini-invLK-ODE}
\frac{\di w}{\di t}=\phi(w,t)\quad \text{for a.e.\ } t\in I_s;\qquad w(s)=z.
\end{equation}
Note that in this case solutions do not have to be defined for all~$t\in I_s$. In fact, the set of all $z\in D$ for which the solution to~\eqref{EQ_ini-invLK-ODE} exists up to (and including) a given instant~$t$ coincides with~$v_{s,t}(D)$.

\begin{remark}
Equations~\eqref{EQ_ini-LK-ODE}\,--\,\eqref{EQ_ini-invLK-ODE} are essentially equivalent to each other and are known as \textsl{generalized Loewner\,--\,Kufarev} differential equations; see, e.g.\,\cite{Loewner-survey}. A good account on the \textit{classical} Loewner\,--\,Kufarev equations can be found in \cite[Chapter~6.1]{Pombook75}.
\end{remark}

\subsection{Boundary regular fixed points}\label{sec:BRFP}
For the classical results stated in this section, we refer the reader to~\cite[Sect.\,1.8, 1.9, 10.1, 12.1, and~12.2]{BCD-Book}, see also \cite[Chapter~1]{Abate:book}.
From the dynamics point of view, an important role in the study of self-maps is played by the fixed points. Thanks to the Schwarz\,--\,Pick Lemma, every holomorphic $w\colon \UD\to\UD$ different from~$\id_\UD$ can have at most one fixed point~$\tau\in\UD$. All other fixed points are located on~$\UC$ and should be understood in the sense of angular limits. Namely, $\sigma\in\UC$ is said to be a \textsl{boundary fixed point} for $w\in\Hol(\UD,\UD)$ if the angular limit $\anglim_{z\to\sigma}w(z)$ exists and coincides with~$\sigma$. It is known that the angular derivative
$$
 w'(\sigma):=\anglim_{z\to\sigma}\frac{w(z)-\sigma}{z-\sigma}
$$
exists, finite or infinite, at any boundary fixed point~$\sigma$ and that $w'(\sigma)$ is either~$\infty$ or a positive number, see e.g.~\cite[Proposition~4.13 on p.\,82]{Pommerenke:BB}.
\begin{definition}\label{DF_BRFP}
A boundary fixed point~$\sigma\in\UC$ of a holomorphic self-map ${w\colon \UD\to\UD}$ is said to be \textsl{regular} (\textsl{BRFP} for short) if $w'(\sigma)\neq\infty$.
\end{definition}

Using the Cayley map $\UD\ni z\mapsto H(z):=(1+z)/(1-z)\in\UH$, the above definitions can be extended to the holomorphic self-maps of the half-plane~$\UH$. In fact, for $\sigma\in\partial\UH$ different from~$\infty$ this leads to literally the same definitions. In case $\sigma=\infty$, the angular derivative $w'(\infty)$ should be understood in the sense of Carath\'eodory as defined in the following version of Julia's Lemma.
\begin{result}[\protect{see e.g. \cite[\S26]{Valiron:book}}]\label{TH_Julia-half-plane}
For any $f\in\Hol(\UH,\overline\UH)$, the limit
$$
f'(\infty):=\anglim_{\zeta\to\infty}\frac{f(\zeta)}{\zeta}
$$
exists finitely. Moreover,
$$
f'(\infty)=\inf_{\zeta\in\UH}\frac{\Re f(\zeta)}{\Re \zeta}\ge0.
$$
In particular, $$f(\zeta)=f'(\infty)\zeta+g(\zeta)\vphantom{\textstyle\int\limits_0^1}$$ for all ${\zeta\in\UH}$ and some $g\in\Hol(\UH,\overline\UH)$ satisfying $g'(\infty)=0$.
\end{result}
If $f\in\Hol(\UH,\UH)$ has a boundary fixed point at~$\infty$, then $w:=H^{-1}\circ f\circ H$ has a boundary fixed point at~$1$ and ${w'(1)=1/f'(\infty)}$. Therefore, $\infty$ is a BRFP for $f\in\Hol(\UD,\UD)$ if and only if~$f'(\infty)\neq0$.

\begin{remark}
It follows easily from \cite[Theorem~1]{CoDiPo06} that $0$ is a BRFP for $w\in\Hol(\UH,\UH)$ if and only if ${w(x)\to0}$ as ${\R\ni x\to0^+}$ and the limit~$\lim_{x\to 0^+}w'(x)$ exists finitely. Similarly, $\infty$ is a BRFP for $w\in\Hol(\UH,\UH)$ if and only if $\lim_{x\to +\infty}w'(x)$ is different from zero.
\end{remark}

If $w\in\Hol(D,D)$, where $D\in\{\UD,\UH\}$, has no fixed points in~$D$, then according to the classical Denjoy\,--\,Wolff Theorem, there exists a unique BRFP ${\tau\in\partial D}$ of~$w$ such that ${w'(\tau)\le 1}$ if ${\tau\neq\infty}$ and ${w'(\tau)\ge1}$ if ${\tau=\infty}$. Moreover, as ${n\to+\infty}$, the iterates
\begin{equation}\label{EQ_ite-to-tau}
 w^{\circ n}:=\underbrace{w\circ w\circ\ldots\circ w}_{n \text{~times}}~\longrightarrow~\tau
\end{equation}
locally uniformly in~$D$. The point $\tau$ is called the \textsl{Denjoy\,--\,Wolff point} of~$w$ or the \textsl{DW-point} for short. The limit behaviour of holomorphic self-maps with interior fixed points is similar. Namely, if $\tau\in D$ is a fixed point of $w\in\Hol(D,D)$ and if $w$ is not an automorphism of~$D$, then~\eqref{EQ_ite-to-tau} holds. In this case we also refer to~$\tau$ as the \textsl{Denjoy\,--\,Wolff point} of~$w$.

Now if $(v_t)\subset\Hol(D,D)$, $D\in\{\UD,\UH\}$, is a non-trivial one-parameter semigroup, then $v_t$'s different from~$\id_D$ share the same BRFPs and in particular, the same DW-point. According to a well-known result of Berkson and Porta~\cite{BP78}, see also \cite[Theorem~10.1.10]{BCD-Book} for another approach, in case $D=\UD$ the infinitesimal generator $\phi$ of a one-parameter semigroup $(v_t)$ can be written as
\begin{equation}\label{EQ_BP}
\phi(z)=(z-\tau)(1-\overline\tau z)p(z),\quad{z\in\UD},
\end{equation}
where $\tau$ is the DW-point of~$(v_t)$ and $p$ is a suitable (uniquely defined) holomorphic function in~$\UD$ with ${\Re p\ge0}$. Conversely, if $\phi\in\Hol(\UD)$ is given by~\eqref{EQ_BP} with some $\tau\in\overline\UD$ and some $p\in\Hol(\UD,\overline\UH)$, then $\phi$ is the infinitesimal generator of a one-parameter semigroup with the DW-point at~$\tau$.

\begin{remark}\label{RM_BRFP-semigroups}
It is known~\cite[Theorem~1]{CoDiPo06} that $\sigma\in\partial\UD$ is a BRFP of a one-parameter semigroup $(v_t)\subset\Hol(\UD,\UD)$ if and only if the limit
\begin{equation}\label{EQ_anglim-BRCP}
 \lambda:=\anglim_{z\to\sigma}\frac{\phi(z)}{z-\sigma}
\end{equation}
is finite, and in such a case ${v_t'(\sigma)=e^{-\lambda t}}$ for
all~${t\ge0}$. In particular, ${\lambda\in\Real}$.
\end{remark}
\begin{remark}\label{RM_BRFP-radial}
It is worth mentioning that the angular limit~\eqref{EQ_anglim-BRCP} exists, \textit{finite or infinite}, for any infinitesimal generator~$\phi$ in the unit disk~$\UD$, see e.g.~\cite[p.\,330]{BCD-Book}. Therefore, \eqref{EQ_anglim-BRCP}~can be replaced with the corresponding radial limit $\lim_{r\to1^-}\overline\sigma\phi(\sigma r)/(r-1)$.
\end{remark}
\begin{remark}\label{RM_critical-via-radial-limits}
It is also known, see e.g. \cite[Proposition~12.2.4]{BCD-Book}, that the limit~\eqref{EQ_anglim-BRCP} is finite if and only if $\lim_{r\to 1^-}\phi(r\sigma)=0$ and the limit $\phi'(\sigma):=\lim_{r\to 1^-}\phi'(r\sigma)$ exists finitely. In such a case ${\lambda=\phi'(\sigma)}$.
\end{remark}
For the non-autonomous case the following extension of~\cite[Theorem~1]{CoDiPo06} holds.
\begin{result}[\protect{\cite[Theorem 1.1]{BRFP2015}}] \label{thm:boundary_derivative}
 Let $(w_{s,t})_{(s,t)\in\Delta(I)}\subset\Hol(\UD,\UD)$ be an absolutely continuous evolution family with associated Herglotz vector field ${\phi\colon \UD\times I\to\Complex}$. Let ${\sigma\in\UC}$. Then the following two conditions are equivalent.
 \begin{itemize}
 \item[\rm (A)] For every $(s,t)\in\Delta(I)$, $\sigma$ is a BRFP of $w_{s,t}$.\vskip.7ex
 \item[\rm (B)] The limit $\lambda(t):=\anglim_{z\to\sigma}\phi(z,t)/(z-\sigma)$ is finite for a.e. $t\ge0$,\\  and $\lambda\in\Lloc(I)$.
 \end{itemize}
If the above conditions are satisfied, then for all $(s,t)\in\Delta(I)$,
$$
\log w'_{s,t}(\sigma)=-\int_s^t\lambda(u)\di u.
$$
\end{result}

\begin{remark}\label{RM_BP-in-H}
Remark~\ref{RM_conformal-trans-of-generators} allows one to extend the Berkson\,--\,Porta representation formula~\eqref{EQ_BP} to infinitesimal generators in~$\UH$. Namely, $\phi\colon \UH\to\C$ is an infinitesimal generator of a one-parameter semigroup in $\Hol(\UH,\UH)$ with the DW-point at~${\tau\in\overline{\UH}}$ if and only if
\begin{subequations}\label{EQ_Berkson-Porta-UH}
 \begin{align}
 \label{EQ_BP_H-finite}
  \phi(\zeta)&=(\zeta-\tau)(\zeta+\overline\tau)P(\zeta) && \text{if~$~\tau\in\UH\cup i\,\Real$},\\
 \label{EQ_BP_H-infty}
  \phi(\zeta)&=-P(\zeta) && \text{if~$~\tau=\infty$},
 \end{align}
\end{subequations}
for all $\zeta\in\UH$, where $P$ is a holomorphic function in~$\UH$ with~$\Re P\ge0$.
\end{remark}

\subsection{Bernstein functions}\label{SS_Bernstein} We start with the following classical definition.
\begin{definition}
By a \textsl{Bernstein function} we mean a non-negative function $f$ of class $C^{\infty}$ in $(0,+\infty)$ such that $(-1)^{n+1}f^{(n)}(x)\ge0$ for all $x>0$ and all~$n\in\Natural$.
\end{definition}

In the following theorem we collect some important basic facts, which can be found together with proofs in literature on Bernstein functions, in particular, in the monograph~\cite[Chapter~3]{SSV12}.
\begin{result}\label{TH_properties-of-BF}
The following statements hold.
\begin{itemize}
\item[\rm (A)] $f\colon (0,+\infty)\to\Real$ is a Bernstein function if and only if it admits the following representation:
    \begin{equation}\label{EQ_LKh-formula}
    f(x)=\alpha + \beta x~ +\! \int\limits_{(0,+\infty)}\! (1-e^{-\lambda x}) \,\rho(\di\lambda)  \qquad \,\text{for all~$~x>0$,}
    \end{equation}
    where $\alpha,\beta \ge0$ and $\rho$ is a non-negative Borel measure on $(0,\infty)$ satisfying $\int_0^{+\infty}\min\{\lambda,1\}\,\rho(\di\lambda)<+\infty$.\medskip
\item[\rm (B)] Every Bernstein function $f$ has a (unique) holomorphic extension ${f_*\colon \UH\to\Complex}$ with ${\Re f_*\ge0}$. Moreover, $f_*$ extends continuously to $i\,\Real$ and representation~\eqref{EQ_LKh-formula} still holds if $\,f$ is replaced by $f_*$ and ${x>0}$ by an arbitrary $z\in\UH$.\medskip
\item[\rm (C)] For the coefficients $\alpha$ and $\beta$ in representation~\eqref{EQ_LKh-formula} we have
    $$\alpha=\lim_{x\to0^+}f(x)=\lim_{\UH\ni z\to 0}f_*(z)\quad\text{and}\quad \beta=\lim_{x\to+\infty}\frac{f(x)}{x}=f_*'(\infty).$$
\item[\rm (D)] If $f\not\equiv0$ and $g$ are Bernstein functions, then $\,g\circ f$ is also a Bernstein function and trivially $(g\circ f)_*=g_*\circ f_*$.\medskip
\item[\rm (E)] Bernstein functions form a convex cone; namely, for any two Bernstein functions $f_1$, $f_2$ and any constant $a,b\ge0$, $af_1+bf_2$ is again a Bernstein function.\medskip
\item[\rm (F)] If a sequence of Bernstein functions converges pointwise in~$(0,+\infty)$ to a function~$f$, then $f$ is  a Bernstein function.
\end{itemize}
\end{result}

In what follows we will not make any notational distinction between a Bernstein function~$f$ and its holomorphic extension~$f_*$.  The class of all $v\in\Hol(\UH,\C)$ such that $v|_{(0,+\infty)}$ is a Bernstein function and not identically zero will be denoted by $\BF$. According to assertions (B), (D) and~(F) in the above theorem, $\BF$ is a topologically closed semigroup in $\Hol(\UH,\UH)$. Taking into account the uniqueness of holomorphic extension, we will refer to elements of~$\BF$ as Bernstein functions as well.

In what follows we will need another definition, closely related to the notion of a Bernstein function.
\begin{definition}
  Let $I\subset\Real$ be an open interval.
  A function $\varphi\colon I\to\Real$ is said to be \textsl{completely monotone} in~$I$, if it is of class $C^\infty$ and ${(-1)^k\varphi^{(k)}(x)\ge0}$ for all $x\in I$ and all $k=0,1,2,\ldots$

  Similarly, a function $\psi\colon I\to\Real$ is said to be \textsl{absolutely monotone} in~$I$ if it is of class $C^\infty$ and ${\psi^{(k)}(x)\ge0}$ for all $x\in I$ and all $k=0,1,2,\ldots$
\end{definition}
\begin{remark}
By the very definition, a non-negative differentiable function $f$ on $(0,+\infty)$ is a Bernstein function if and only if $\psi:=f'$ is completely monotone on~$(0,+\infty)$.  Moreover, obviously $\varphi$ is a completely monotone in an open interval~$I$ if and only if $\psi(x):=\varphi(-x)$ is absolutely monotone in the interval $J:=\{{-x}\colon {x\in I}\}$. Another simple observation is that according to Leibniz's rule for higher order derivatives, the product of two absolutely monotone functions is again absolutely monotone.
\end{remark}
\begin{remark}\label{RM_Faa-di-Bruno}
In a similar way, it follows immediately from Fa\`{a} di Bruno's formula for the derivatives of a composite function that if $g\colon I\to J$, where $I$ and $J$ are two open intervals in~$\R$, is such that $g'$ is absolutely monotone in~$I$, and if $\psi$ is absolutely monotone in~$J$, then the composition $\psi\circ g$ is absolutely monotone in~$I$.
\end{remark}
\begin{remark}\label{RM_composition}
The above remarks imply a slight generalization of assertion~(D) in Theorem~\ref{TH_properties-of-BF}. Namely, if $f_1$ is a Bernstein function and if $f_2$ is a non-negative differentiable function such that $f_2'$ is completely monotone in an open interval~$J$ containing~$f_1\big((0,+\infty)\big)$, then $f_2\circ f_1$ is a Bernstein function: simply apply Remark~\ref{RM_Faa-di-Bruno} with $g:=g_1$, $I:=(-\infty,0)$, and  $\psi:=g_2'$, where $g_j(x):=-f_j(-x)$, $j:=1,2$.
\end{remark}

%
%=============================================
%
\section{Results}\label{S_results}
%
%=============================================
%
\subsection{Infinitesimal generators in topologically closed semigroups}\label{SS_one-param-in-semigroups}
It is well-known that infinitesimal generators of one-parameter semigroups in $\Hol(D,D)$ form a closed convex cone, see e.g.~\cite[Corollary~10.2.7, p.\,287]{BCD-Book}. For various choices of a semigroup $\U\subsetneq\Hol(D,D)$, the cone $\Gen(\U)$ formed of all infinitesimal generators of one-parameter semigroups contained in~$\U$ has been explicitly characterized, see e.g.\,\cite{Goryainov-survey}. In each case we are aware of, the cone $\Gen(\U)$ turns out to be convex. However, we have not been able to find in the literature any general results in this respect.
By this reason, below we give a proof of the following basic theorem.

\begin{theorem}\label{TH_generators-convex-cone}
Let $\DC$ be a simply connected domain and $\,\U$ a topologically closed semigroup of holomorphic self-maps of~$D$.
Then the set $\Gen(\U)$ formed by all infinitesimal generators of one-parameter semigroups contained in~$\U$ is a topologically closed (real) convex cone in~$\Hol(D,\C)$.
\end{theorem}
Since by Theorem~\ref{TH_properties-of-BF}, $\BF$ is a topologically closed semigroup in $\Hol(\UH,\UH)$, the above theorem directly implies the following
\begin{corollary}\label{CR_BernsteinGen-conv-cone}
The set $\,\Gen(\BF)$ is a topologically closed convex cone in~$\Hol(\UH,\C)$.
\end{corollary}

\begin{proof}[\proofof{Theorem~\ref{TH_generators-convex-cone}}]
The fact that $\Gen(\U)$ is a cone is trivial. Indeed, if $\phi\in\Gen(\U)$ generates a one-parameter semigroup~$(v_t^\phi)$, then for any $\alpha\ge0$ the function $\alpha\phi$ is the infinitesimal generator of the one-parameter semigroup $(\tilde v_t)$ formed by the mappings ${\tilde v_t:=v_{\alpha t}^\phi}$.

To show that $\Gen(\U)$ is topologically closed, consider an arbitrary sequence ${(\phi_n)\subset \Gen(\U)}$ converging locally uniformly in~$D$ to some $\phi\in\Hol(D,\C)$. By \cite[Theorem~10.5.6, p.\,300]{BCD-Book}, $\phi$ is also an infinitesimal generator in~$D$ and, for each $t\ge0$, we have ${v_t^n\to v_t}$ in $\Hol(D,D)$ as $n\to+\infty$, where $(v_t^n)$ and $(v_t)$ stand for the one-parameter semigroups generated by~$\phi_n$ and~$\phi$, respectively. By the hypothesis, $(v_t^n)\subset\U$ and $\U$ is closed in $\Hol(D,D)$. Therefore, $(v_t)\subset\U$ and hence $\phi\in\Gen(\U)$. This means that $\Gen(\U)$ is a closed subset of~$\Hol(D,\C)$.

To complete the proof it remains to show that ${\phi:=\phi_1+\phi_2\in\Gen(\U)}$ provided that $\phi_1,{\phi_2\in\Gen(\U)}$.
By \cite[Corollary~10.2.7, p.\,287]{BCD-Book}, $\phi$~is an infinitesimal generator in~$D$. Denote by $(v_t^1)$, $(v_t^2)$, and $(v_t)$ the one-parameter semigroups generated by~$\phi_1$, $\phi_2$, and~$\phi$, respectively. The Trotter product formula for one-parameter semigroups, see e.g. \cite[Corollary~10.6.2, p.\,304]{BCD-Book} or \cite[Corollary~4]{ProdFormula}, states that for each ${t\ge0}$,
$$
v_t=\lim_{n\to+\infty}\big(v^1_{t/n}\circ v^2_{t/n}\big)^{\circ n},
$$
where the limit exists in the topology of $\Hol(D,D)$. Since by the hypothesis, $\U$ is a closed subset of~$\Hol(D,D)$ and all finite compositions of its elements again belong to~$\U$, it follows that ${(v_t)\subset\U}$, i.e. $\phi\in\Gen(\U)$.
\end{proof}

\subsection{Evolution families in topologically closed semigroups}\label{SS_EF-in-semigroups}
A natural question arises in connection with the correspondence discussed at the end of Section~\ref{SS_REF}: how are the properties of the elements in a (reverse) evolution family related to those of the associated Herglotz vector field? For a particular property related to the boundary behaviour at a given point (namely, for having a boundary regular fixed point) the answer was given as Theorem \ref{thm:boundary_derivative}. The main new result of this section is a general statement of a similar kind.

\begin{theorem}\label{TH_EvolFam-HVF}
Let $\DC$ be a simply connected domain and $\,\U$ a topologically closed semigroup of holomorphic self-maps of~$D$. Let $(w_{s,t})$ be a (reverse) evolution family in~$D$ with associated Herglotz vector field~$\phi$. Then the following two conditions are equivalent:
\begin{itemize}
\item[(i)] $w_{s,t}\in\U$ for all $(s,t)\in\Delta(I)$;\vskip.5ex
\item[(ii)] $\phi(\cdot,t)\in\Gen(\U)$ for a.e. $t\in I$.
\end{itemize}
\end{theorem}
\begin{remark}
The assumption that $\U$ is topologically closed plays an important role in the above theorem. Compare, e.g., with Theorem~\ref{thm:boundary_derivative} in Section~\ref{sec:BRFP}, which concerns holomorphic self-maps with a given BRFP. They form a semigroup, but it is not a closed subset in~$\Hol(D,D)$. As a result, an additional condition has to be imposed on a Herglotz vector field in order to ensure that the corresponding evolution family is contained in the semigroup in question.
\end{remark}
Nevertheless, since by Theorem~\ref{TH_properties-of-BF}, Bernstein functions different from the identical zero form a topologically closed semigroup in $\Hol(\UH,\UH)$, Theorem~\ref{TH_EvolFam-HVF} implies directly the following corollary.

\begin{corollary}\label{CR_Herglotz-VF-for-Bernstein-evol-fam}
  An absolutely continuous evolution family $(w_{s,t})_{(s,t)\in\Delta(I)}$ in~$\UH$ is contained in the class~$\BF$ if and only if the Herglotz vector field $\phi$ associated to~$(w_{s,t})$ satisfies the following condition: $\phi(\cdot,t)\in\Gen(\BF)$ for a.e.~$t\in I$.
\end{corollary}

The proof of (ii)\,$\Rightarrow$\,(i) in Theorem~\ref{TH_EvolFam-HVF} makes use of the following rather simple observation, which however we have not met explicitly stated elsewhere.
\begin{proposition}\label{PR-HVF-intergrability}
For every Herglotz vector field $\phi$ in a simply connected domain $\DC$, there exists a non-negative function ${M\in\Lloc(I)}$ with the following property: for any compact set $K\subset D$ there is a constant $C_K>0$ such that
$$
\max_{z\in K}|\phi(z,t)|\le C_K M(t)\quad\text{for a.e.~$t\ge0$}.
$$
\end{proposition}
\begin{proof}
Thanks to Remark~\ref{RM_conformal-trans-of-HVFs}, we may suppose that $D=\UD$. Then by the non-autonomous extension of the Berkson\,--\,Porta formula \cite[Theorem~4.8]{BCM1},
\begin{equation}\label{EQ_BP-non-aut}
 \phi(z,t)=\big(z-\tau(t)\big)\big(1-\overline{\tau(z)}z\big)p(z,t) \quad \text{for all~$z\in\UD$ and
 a.e.~$t\in I$},
\end{equation}
where $\tau\colon I\to\overline\UD$ is a measurable function and $p\colon \UD\times I\to\C$ satisfies the following two conditions:
\begin{itemize}
\item[(a)] for a.e. $t\in I$, the function $p(\cdot,t)$ is holomorphic in~$\UD$ and $\Re p(\cdot,t)\ge0$;
\item[(b)] for any $z\in\UD$, the function $p(z,\cdot)$ is locally integrable on~$I$.
\end{itemize}

Set $M(t):=|p(0,t)|$ for all ${t\in I}$. Then by~(b), ${M\in\Lloc(I)}$. Moreover, thanks to~(a) with the help of the Harnack inequality,  see e.g. \cite[ineq.\,(11) on p.\,40]{Pombook75}, for a.e. ${t\in I}$ and all $z\in\UD$ we have
\begin{align}\label{EQ_p-est}
|p(z,t)|&\le |\Im p(0,t)|\,+\,\frac{1+|z|}{1-|z|}\Re p(0,t)\\\notag &\le\frac{1+|z|}{1-|z|}\big(|\Im p(0,t)|+\Re p(0,t)\big)\le\frac{1+|z|}{1-|z|}\sqrt{2}M(t).
\end{align}

Now the conclusion of the proposition follows easily from~\eqref{EQ_BP-non-aut} and~\eqref{EQ_p-est}.
\end{proof}

\begin{proof}[\proofof{Theorem~\ref{TH_EvolFam-HVF}}]
First of all, recall that according to Remark~\ref{RM_REF-EF}, if $(v_{s,t})$ is reverse evolution family over~$I$ with associated Herglotz vector~$\phi$, then for any ${\Tau\in I}$, the family $(w_{s,t})_{0\le s\le t\le \Tau}$ formed by the functions ${w_{s,t}:=v_{\Tau-t,\Tau-s}}$ is an evolution family over~$[0,\Tau]$. It is easy to see that $[0,\Tau]\ni t\mapsto \phi(\cdot,\Tau-t)$ is the Herglotz vector field associated with~$(w_{s,t})$.
It is, consequently,  sufficient to prove the theorem for the case of evolution families.

\medskip
\noindent{\sc Proof of} (i)\,$\Rightarrow$\,(ii).
By \cite[Theorem 3.6]{CDG-AnnulusI}, for each evolution family $(w_{s,t})$ in~$D$, there exists a null-set~$N\subset I$ such that for any $s\in [0,T)$ and any~$t_0\in I_s\setminus N$ the limit
$$
\lim_{I_s\ni t\to t_0}\frac{w_{s,t_0}-w_{s,t}}{t-t_0}
$$
exists in $\Hol(D,\C)$ and coincides with $z\mapsto \phi\big(w_{s,t_0}(z),t_0\big)$.
In particular, if $s\in [0,T)\setminus N$, then taking $t_0:=s$ we see that $$\frac{\id_D-w_{s,t}}{t-s}\to \phi(\cdot,s)$$ locally uniformly in~$D$ as ${t\to s^+}$. It follows, see e.g. \cite[Theorem~10.6.1]{BCD-Book} or \cite[Theorem~3]{ProdFormula}, that $\phi(\cdot, s)$ is an infinitesimal generator in~$D$ and the corresponding one-parameter semigroup $(v_t^s)$ is given by the formula
\begin{equation}\label{EQ_prod}
v_t^s=\lim_{n\to+\infty}w^{\circ n}_{s,s+t/n} \quad \text{for all $t\ge0$},
\end{equation}
where the limit is attained locally uniformly in~$D$. By the hypothesis, $\U$ is a topologically closed semigroup. Therefore, if ${(w_{s,t})\subset\U}$, then thanks to~\eqref{EQ_prod} we also have ${(v_t^s)\subset\U}$. Thus, assertion~(i) implies that $\phi(\cdot,s)\in\Gen(\U)$ for a.e. ${s\in[0,T)}$, i.e. that~(ii) holds.

\medskip
\noindent{\sc Proof of} (ii)\,$\Rightarrow$\,(i).  By Theorem~\ref{TH_generators-convex-cone}, $\Gen(\U)$ is a topologically closed convex cone.
By Proposition~\ref{PR-HVF-intergrability}, for any compact set $K\subset D$, there exists $C_K>0$ such that
\begin{equation}\label{EQ_G-est}
\max_{z\in K}|\phi(z,t)|\le C_K M(t)\quad\text{for a.e. $t\in I$},
\end{equation}
where $M\colon I\to [0,+\infty)$ is a locally integrable  function (not depending on~$K$).

Let $g(t):=\int_{0}^t M(s)\di s$ for all~${t\ge0}$. Clearly, we may suppose that $M(t)\ge1$ for all $t\in I$; otherwise we would replace the function~$M$ by $t\mapsto{\max\{M(t),1\}}$.
Then both ${g\colon I\to J:=g(I)\subset\R}$ and ${f:=g^{-1}\colon J\to I}$ are strictly increasing and locally absolutely continuous. Taking into account~\eqref{EQ_G-est}, it follows that the formula
$$
\psi(z,\theta):=\frac{\phi(z,f(\theta))}{M(f(\theta))}\quad\text{for all $z\in D$ and $\theta\in J$},
$$
defines a Herglotz vector field $\psi:D\times J\to\C$. According to Theorem~\ref{TH_mainBCM1}, the evolution family associated with $\psi$ can be constructed via solutions ${\omega=\omega(\theta)}$ to the initial value problem
$$
\frac{\di \omega(\theta)}{\di \theta}+\psi\big(\omega(\theta),\theta\big)=0\quad \text{for a.e.\ } \theta\ge \eta,~\theta\in J;\quad \omega(\eta)=z.
$$
The variable change $\theta=g(t)$ relates the trajectories of the Herglotz vector field~$\psi$ to those of~$\phi$. Using this fact, it is easy to see that the evolution family $(\omega_{\eta,\theta})$ associated with~$\psi$ is given by $\omega_{\eta,\theta}=w_{f(\eta),f(\theta)}$ for all $(\eta,\theta)\in\Delta(J)$. Therefore, it is sufficient to show that $(\omega_{\eta,\theta})\subset \U$.

From~\eqref{EQ_G-est} it follows that
\begin{equation}\label{EQ_H-est}
\max_{z\in K}|\psi(z,\theta)|\le C_K\quad\text{for a.e. $\theta\in J$}
\end{equation}
and any compact set $K\subset D$.

Consider the sequence $(\psi_n)$ of functions on $D\times J$ defined by
\begin{equation}\label{EQ_H_n}
\psi_n(z,\theta):=2^n\!\!\int\limits_{\eta(\theta,n)}^{\eta(\theta,n)+2^{-n}}\!\!\psi(z,\eta)\,\di\eta,\quad \theta\in J,~n\in\N,
\end{equation}
where $\eta(\theta,n):=2^{-n}\lfloor2^n \theta\rfloor$ and $\lfloor\cdot\rfloor$ stands for the integer part.

Denote by $\Gen_0$ the set of all $\varphi\in\Gen(\U)$ such that $\sup_K|\varphi|\le C_K$ for every compact set $K\subset D$. Recall that by Theorem~\ref{TH_generators-convex-cone}, $\Gen(\U)$ is convex and topologically closed. Therefore, $\Gen_0$ is a compact convex set in the locally convex Hausdorff space $\Hol(D,\C)$. By~\eqref{EQ_H-est} the integrand in~\eqref{EQ_H_n} belongs to~$\Gen_0$ for a.e.~${t\ge0}$, and by \cite[Lemma~4.7]{BCM1} the map ${\theta\mapsto \psi(\cdot,\theta)\in\Hol(D,\C)}$ is measurable. Therefore, for any fixed ${n\in\Natural}$ and $\theta\in J$, the integral in~\eqref{EQ_H_n} can be understood as a weak integral of~${\eta\mapsto \psi(\cdot,\eta)}$, and moreover, its value belongs to $2^{-n}\Gen_0$, see e.g. \cite[Theorem~1 in \S18.1.3]{Kadets}. Hence ${\psi_n(\cdot,\theta)\in\Gen_0\subset\Gen(\U)}$ for all ${\theta\in J}$. In its turn, the latter implies that $\psi_n$'s are Herglotz vector fields in~$D$.

Note that the elements of the evolution families $(\omega^n_{\eta,\theta})$ associated with the Herglotz vector fields~$\psi_n$ are finite compositions of mappings from one-parameter semigroups whose infinitesimal generators belong to~$\Gen_0$. Therefore, $(\omega^n_{\eta,\theta})\subset \U$ and it only remains to show that for each ${(\eta,\theta)\in\Delta(J)}$ fixed,
$$
 \omega_{\eta,\theta}^n\to\omega_{\eta,\theta}~\text{~in $~\Hol(D,\C)$}\quad\text{ as $~n\to+\infty$.}
$$
Since self-maps of~$D$ form a normal family, it is sufficient to prove that $\omega_{\eta,\theta}^n\to\omega_{\eta,\theta}$ pointwise in~$D$.

We claim that for a.e.\,$\theta\in J$, $\psi_n(\cdot,\theta)\to \psi(\cdot,\theta)$ in $\Hol(D,\C)$ as ${n\to+\infty}$.
Indeed, consider a sequence $(z_k)\subset D$ having an accumulation point in~$D$. By compactness of~$\Gen_0$, it is only necessary to prove that for each $k\in\Natural$ and for a.e.\,$\theta\in J$, $\psi_n(z_k,\theta)\to \psi(z_k,\theta)$ as ${n\to+\infty}$. The latter holds for any Lebesgue point~$\theta$, i.e. for any point $\theta\in J$ at which the derivative of $F_k(\theta):=\int_0^\theta \psi(z_k,\eta)\di\eta$ exists and equals $\psi(z_k,\theta)$. Since the Lebesgue points have full measure, our claim is proved.

Fix now a compact set $K\subset D$. As we have just proved,
$$
 \delta^K_n(\theta):=\max_{z\in K}|\psi_n(z,\theta)-\psi(z,\theta)|\to 0\quad\text{as~$~n\to+\infty$}
$$
for a.e. $\theta\in J$. Since $\psi_n(\cdot,\theta)$ and $\psi(\cdot,\theta)$ belong to $\Gen_0$ for all $n\in\Natural$ and a.e.\,$\theta\in J$, we have
$0\le \delta^K_n(\theta)\le 2 C_K$ for a.e. $\theta\in J$.
Therefore, thanks to Lebesgue's Dominated Convergence Theorem, for each compact interval ${Y\subset J}$ and each function ${h:Y\to\C}$ of class $L^\infty$, we have
$$
\left|\int_Y h(\theta)\psi_n(z,\theta)\,\di\theta \,-\, \int_Y h(\theta)\psi(z,\theta)\,\di\theta\right|\le\int_Y|h(\theta)|\,\delta^K_n(\theta)\,\di\theta\,\to\,0~\text{~as~$n\to+\infty$}
$$
for all~$~z\in K$.
The fact that $\omega_{\eta,\theta}^n\to\omega_{\eta,\theta}$ as ${n\to+\infty}$ follows now from general results on ODEs, see e.g. \cite[Lemma~I.37 on p.\,38]{OliverDiss} or~\cite[Lemma~3.1]{IkkeiJAM}.
\end{proof}

\subsection{Bernstein generators}\label{SS_Bernstein-generators}
By a \textsl{Bernstein generator} we mean an infinitesimal generator of a one-parameter semigroup whose elements are Bernstein functions. In the notation introduced in Preliminaries, the class of all Bernstein generators is exactly~$\Gen(\BF)$. It is worth mentioning that in the literature on stochastic processes, see e.g. \cite[Sect.\,12.1]{Kyp14}, Bernstein generators appear in connection to continuous-state branching processes; in this probabilistic context, they are referred to as \textsl{branching mechanisms}.

\subsubsection{Statement of results} An integral representation of Bernstein generators was found in 1968 by Silverstein~\cite[Theorem 4]{Sil68}, see also \cite[Lemma~1.3]{KW1971}.
We establish an essentially equivalent way to characterize Bernstein generators.

\begin{theorem}\label{TH_BG}
A function $\phi\in\Hol(\UH,\C)$ is a Bernstein generator if and only if it satisfies the following three conditions:
\begin{itemize}
  \item[(i)] $\phi\big((0,+\infty)\big)\subset\Real$;
  \item[(ii)] the limit $\phi(0):=\lim_{\Real\ni x\to0^+} \phi(x)$ exists and belongs to $(-\infty,0]$;
  \item[(iii)] $\phi''$ is completely monotone in~$(0,+\infty)$.
\end{itemize}
\end{theorem}
Following Kyprianou~\cite[Theorem~12.1 on p.\,336]{Kyp14}, Silverstein's representation can be written as follows. A function $\phi\colon \UH\to\C$ is a Bernstein generator if and only if there exist $a \in \R$, $q, b \ge0$ and a Borel non-negative measure $\pi$ on $(0,+\infty)$ with $\int_0^{+\infty} \min\{x^2,1\} \,\pi(\di x)<+\infty$ such that
\begin{equation}\label{EQ_Silver-representation}
\phi(\zeta) = -q + a \zeta + b \zeta^2 + \int_0^{+\infty} \!\!\big(e^{-\zeta x} - 1 + \zeta x\,\ind_{(0,1)}(x)\big) \,\pi(\di x) \quad  \text{for all~}~\zeta\in \UH
\end{equation}
(where we have extended formula~\cite[(12.7)]{Kyp14} holomorphically from ${(0,+\infty)}$ to~$\UH$).

If $\phi$ is given by~\eqref{EQ_Silver-representation}, then clearly $\phi\in\Hol(\UH,\C)$ and satisfies conditions (i) and~(ii) in Theorem~\ref{TH_BG}. Moreover,~\eqref{EQ_Silver-representation} implies that $\phi''$ is the Laplace transform of a non-negative Borel measure $\mu$ on ${[0,+\infty)}$ given by
\begin{equation}\label{EQ_measure-mu}
\mu(\{0\}):=2b,\quad \mu|_{(0,+\infty)}(\di x):=x^2\pi(\di x).
\end{equation}
By Bernstein's Theorem, see e.g. \cite[Theorem~1.4 on p.\,3]{SSV12}, $\phi''$ is completely monotone on ${(0,+\infty)}$, i.e. condition~(iii) is also satisfies.

Conversely, suppose $\phi\in\Hol(\UH,\C)$ and conditions (i)\,--\,(iii) hold. Then again by Bernstein's Theorem, $\phi''$ is the Laplace transform of some non-negative $\sigma$-finite Borel measure $\mu$ on ${[0,+\infty)}$. Recovering $\phi'$ and $\phi$ with the help of the formulas $\phi'(x)=\phi'(1)+\int_1^x\phi''(y)\di y$, $\phi(x):=\phi(0)+\int_0^x\phi'(y)\di y$ for all $x>0$, and taking into account that by (ii) the integral $\int_0^1\phi'(y)\di y$ converges, we see that:
\begin{itemize}
 \item[(a)] the Borel measure $\pi$ on $(0,+\infty)$ defined by $\pi(\di\lambda):=\lambda^{-2}\mu|_{(0,+\infty)}(\di\lambda)$ satisfies  $\int_{0}^{+\infty}\min\{\lambda^2,1\}\,\pi(\di\lambda)<+\infty$;\smallskip
 \item[(b)] representation~\eqref{EQ_Silver-representation} holds for all $\zeta\in(0,+\infty)$~--- and hence for all~$\zeta\in\UH$~--- with the measure~$\pi$ defined in~(a), ${q:=-\phi(0)}$, ${b:=\mu(\{0\})/2}$,
     $$
        a:=\phi'(1)-\mu(\{0\})+\int_{(0,1)}\!\lambda(e^{-\lambda}-1)\,\pi(\di\lambda)+\int_{[1,+\infty)}\!\lambda e^{-\lambda}\,\pi(\di\lambda).
     $$
\end{itemize}

\begin{remark}
With the help of the monotone convergence theorem, it is easy to see that in the above notation, $$\phi''(+\infty):=\lim_{x\to+\infty}\phi''(x)=\mu(\{0\})=2b.$$
\end{remark}

In this section we will present a complex-analytic proof of Theorem~\ref{TH_BG} independent from representation~\eqref{EQ_Silver-representation}. On the other hand, with the help of the above computation,
\eqref{EQ_Silver-representation} can be adapted for the case of a BRFP at $z:=0$.
\begin{corollary}\label{CR_BG-BRFP0}
A function $\phi\colon \UH\to\C$ is a Bernstein generator associated with a one-parameter semigroup $(v^\phi_t)$ having a BRFP at ${\sigma:=0}$ if and only if $\phi$ admits the representation
\begin{equation}\label{EQ_LeGall-representation}
\phi(\zeta)= c\zeta + b\zeta^2 + \int_0^{+\infty} \!\!\big(e^{-\zeta x} - 1 + \zeta x\big) \,\pi(\di x) \quad  \text{for all~}~\zeta\in \UH,
\end{equation}
where $c\in\Real$, $b\ge0$ and $\pi$ is a non-negative Borel measure on~$(0,+\infty)$ such that
\begin{equation}\label{EQ_BG-BRFP0-int-cond}
\int_{0}^{+\infty}\min\{\lambda^2,\lambda\}\,\pi(\di\lambda)\,<\,+\infty.
\end{equation}

Moreover, if the above necessary and sufficient condition is satisfied, then
\begin{equation}\label{BG-BRFP0}
c=\phi'(0):=\lim_{x\to 0^+}\phi'(x)
\end{equation}
 and $(v_t^\phi)'(0)=e^{-c\,t}$ for all ${t\ge0}$. In particular, ${\sigma=0}$ is the Denjoy\,--\,Wolff point of~$(v_t^\phi)$ if and only if $\phi$ admits representation~\eqref{EQ_LeGall-representation} with ${c\ge0}$.
\end{corollary}
Note that the Bernstein generator~$\phi$ in  Corollary~\ref{CR_BG-BRFP0} admits both representations~\eqref{EQ_Silver-representation} and~\eqref{EQ_LeGall-representation}. As we will see from the proof, the coefficient~$b$ and the measure~$\pi$ in the two representations are the same.

As a representation of a branching mechanism~$\phi$, formula~\eqref{EQ_LeGall-representation} appeared in~\cite[Theorem~1, pp.\,21-22]{LeGall}, \cite[Section~3.1]{Li11}, \cite[Section 2.2]{Li12}.
In the probabilistic terms, $(v^\phi_t)$ has a BRFP at~$0$ if and only if the corresponding continuous-state branching process has finite means, see \cite[Section~12.2.2]{Kyp14}.
%In addition, the branching process with formula \eqref{EQ_LeGall-representation} is said to be critical if $c=0$ and subcritical if $c>0$, see \cite[Definition 12.4]{Kyp14}.

\medskip

In a similar way, the case of a BRFP at~$\infty$ can be treated.
\begin{corollary}\label{CR_BG-BRFP-infty}
A function $\phi\colon \UH\to\C$ is a Bernstein generator associated with a one-parameter semigroup $(v^\phi_t)$ having a BRFP at ${\sigma:=\infty}$ if and only if $\phi$ admits the representation
\begin{equation}\label{EQ_BG-BRFP-infty-representation}
\phi(\zeta)= -q + c\zeta  - \int_0^{+\infty} \!\!\big(1-e^{-\zeta x}\big) \,\pi(\di x) \quad  \text{for all~}~\zeta\in \UH,
\end{equation}
where $q\ge0$, $c\in\Real$, and $\pi$ is a non-negative Borel measure on~$(0,+\infty)$ such that
\begin{equation*}%\label{EQ_BG-BRFP-infty-int-cond}
\int_{0}^{+\infty}\min\{\lambda,1\}\,\pi(\di\lambda)\,<\,+\infty.
\end{equation*}

Moreover, if the above necessary and sufficient condition is satisfied, then
\begin{equation*}%\label{BG-BRFP-infty}
c=\phi'(\infty):=\anglim_{z\to\infty}\frac{\phi(z)}{z}=\lim_{x\to +\infty}\phi'(x)
\end{equation*}
 and $(v_t^\phi)'(\infty)=e^{-c\,t}$ for all ${t\ge0}$. In particular, ${\sigma=\infty}$ is the Denjoy\,--\,Wolff point of~$(v_t^\phi)$ if and only if $\phi$ admits representation~\eqref{EQ_BG-BRFP-infty-representation} with ${c\le0}$.
\end{corollary}

The representation~\eqref{EQ_BG-BRFP-infty-representation} with ${c\le0}$ has a probabilistic interpretation that the corresponding branching process is exactly a process with non-decreasing sample paths because it can be expressed as a randomly time-changed subordinator (with killing) according to the Lamperti transformation (see e.g.\ \cite[Theorem 12.2]{Kyp14}).

\begin{remark}
According to Corollary~\ref{CR_BG-BRFP0}, a function $\phi$ given by representation~\eqref{EQ_LeGall-representation} with $c\ge0$ is an infinitesimal generator of a one-parameter semigroup with the DW-point ${\tau=0}$. It is therefore interesting to compare~\eqref{EQ_LeGall-representation} with the Berkson\,--\,Porta formula~\eqref{EQ_BP_H-finite}: $\phi(\zeta)=\zeta^2P(\zeta)$, $\zeta\in\UH$, where $P$ is the holomorphic function with non-negative real part given by
$$
  P(\zeta)~=~c\,\zeta^{-1}~+\int\limits_{[0,+\infty)}P_*(x\zeta)\,\mu(\di x),\quad P_*(z):=\frac{e^{-z}-1+z}{z^2}\in\UH,\qquad\text{for all~$~\zeta\in\UH$,}
$$
with the Borel measure $\mu$ defined by~\eqref{EQ_measure-mu}.
It is worth mentioning that in the case of arbitrary~$c\in\Real$, the above representation of~$\phi$ corresponds to the extension of the Berkson\,--\,Porta formula due to Bracci, Contreras and D\'\i{}az-Madrigal \cite[Theorem~1.3]{BCD-AK-measure}.
\end{remark}
A similar observation applies to Bernstein generators admitting representation~\eqref{EQ_BG-BRFP-infty-representation}.

\subsubsection{Auxiliary statements and proofs}

\begin{lemma}\label{LM_BG-elementary}
For any $a\in\Real$, $b,q\ge0$, the functions $\phi_1(z):=az$, $\phi_2(z):=bz^2$, and $\phi_3(z):=-q$ are Bernstein generators.
\end{lemma}
\begin{proof}
By explicit integration of the equations $\di w/\di t+\phi_j(w(t))=0$, ${j=1,2,3}$, with the initial condition~$w(0)=z\in\UH$, we obtain $v_t^{\phi_1}(z)=e^{-at}z$, $v_t^{\phi_2}(z)=z/(1+btz)$, and $v_t^{\phi_3}(z)=z+qt$. These are Bernstein functions for all~$t\ge0$.
\end{proof}

\begin{proposition}
\label{PR_Bernstein-func-is-BG}
If $f\in\BF$, then $\phi:=-f\in\Gen(\BF)$.
\end{proposition}
\begin{proof}
Let $f\colon \UH\to\C$ be a Bernstein function. Then $\Re f\ge0$ in~$\UH$ and hence, by Remark~\ref{RM_BP-in-H}, $\phi:=-f$ is an infinitesimal generator in~$\UH$. Clearly, we may suppose that $\phi\not\equiv0$. By Theorem~\ref{TH_Julia-half-plane}, the limit $f'(\infty):=\anglim_{z\to\infty}f(z)/z$ exists and belongs to~$[0,+\infty)$.

Note that we may restrict ourselves to the case $f'(\infty)>0$. To cover the case $f'(\infty)=0$, it would be sufficient to  consider the Bernstein functions $f+\delta\id_\UH$, $\delta>0$, and use the fact that by Corollary~\ref{CR_BernsteinGen-conv-cone}, $\Gen(\BF)$ is topologically closed in $\Hol(\UH,\C)$.

 For any $\varepsilon>0$, Euler's approximations of the one-parameter semigroup~$(v_t^\phi)$ generated by~$\phi$, which are defined recurrently by
$$
 v_0^\varepsilon=\id,\quad v_n^\varepsilon:=v_{n-1}^\varepsilon+\varepsilon f\circ v_{n-1}^\varepsilon,~{n\in\Natural},
$$
are Bernstein functions, because $\BF$ is a convex cone closed w.r.t. compositions and containing~$\id_\UH$, see Theorem~\ref{TH_properties-of-BF}.
Taking into account that by Theorem~\ref{TH_properties-of-BF}\,(F), $\BF$ is topologically closed in~$\Hol(\UH,\UH)$, we only need to show that for any ${t>0}$, $v_n^{t/n}\to v_t^\phi$ locally uniformly in~$\UH$ as~${n\to+\infty}$.

Notice that $v_n^\varepsilon$ is the $n$-th iterate of $\psi_\varepsilon(z):=z+\varepsilon f(z)$. Recall that $f'(\infty)>0$. Therefore,  $\psi_\varepsilon$ is a hyperbolic self-map of~$\UH$ with the DW-point at~$\infty$. Therefore, see e.g. \cite[Chapter~6, \S46]{Valiron:book}, for any ${\varepsilon>0}$ the limit $g_\varepsilon(z):={\lim_{n\to+\infty}\psi_\varepsilon^{\circ n}(z)/|\psi_\varepsilon^{\circ n}(1)|}$ exists for all~${z\in\UH}$; and moreover, $g_\varepsilon\in\Hol(\UH,\UH)$ is a non-constant solution to the functional equation $g_\varepsilon\circ\psi_\varepsilon=\psi'_\varepsilon(\infty) g_\varepsilon$.

Consider the functions
$$
h_\varepsilon:=\frac{\varepsilon}{\log\psi'_\varepsilon(\infty)}\log\circ\,g_\varepsilon,
$$
where $\log$ stands for the branch of the logarithm in~$\UH$ satisfying $\log 1=0$. For each ${\varepsilon>0}$, the function $h_\varepsilon$ solves Abel's functional equation
\begin{equation}\label{EQ_Abel-epsilon}
h_\varepsilon\circ\psi_\varepsilon=h_\varepsilon+\varepsilon.
\end{equation}

Recall that $g_\varepsilon(\UH)\subset\UH$. Hence $h_\varepsilon(\UH)$ lies in the horizontal strip of width $\varepsilon\pi/\log\psi'_\varepsilon(\infty)$ symmetric w.r.t.~$\Real$.
Since $\varepsilon/\log\psi'_\varepsilon(\infty)\to 1/f'(\infty)<+\infty$ as ${\varepsilon\to0^+}$, Montel's normality criterion, see e.g. \cite[\S\,II.7, Theorem~1]{Goluzin}, implies that ${\{h_\varepsilon:0<\varepsilon<1\}}$ is a normal family in~$\UH$.

From Abel's equation~\eqref{EQ_Abel-epsilon} we obtain
$$
  1=\frac{h_\varepsilon(z+\varepsilon f(z))-h_\varepsilon(z)}{\varepsilon}=\int_0^1 h'_\varepsilon(z+t\varepsilon f(z))f(z)\,\di t\quad\text{for any~$z\in\UH$.}
$$
If $h_{\varepsilon_n}\to h$ locally uniformly in~$\UH$ for some sequence $\varepsilon_n\to0^+$ as $n\to+\infty$, then the integral in the above equality tends along the sequence $(\varepsilon_n)$ to $h'(z)f(z)$. Taking into account that $h_\varepsilon(1)=0$ for all ${\varepsilon>0}$, we therefore conclude that
$$
h_\varepsilon(z)\to h(z):=\int_1^z\!\frac{\di\zeta}{f(\zeta)}\quad\text{as~$~\varepsilon\to0^+$.}
$$

The function $h$ is the Koenigs function of~$(v_t^\phi)$, see e.g. \cite[Theorem~10.1.4, p.\,274]{BCD-Book}. In particular, it satisfies Abel's equation $h\circ v_t^\phi=h+t$ for all~$t\ge0$.

On the other hand, if for some fixed $t\ge0$, and some sequence ${(n_k)\subset\Natural}$ tending to~$+\infty$, we have $v^{t/{n_k}}_{n_k}\to v_*$, where $v_*\in\Hol(\UH,\C)\cup\{v\equiv\infty\}$, then in fact, $v_*\in\Hol(\UH,\UH)$ and satisfies Abel's equation $h\circ v_*=h+t$. Indeed, since $\Re f\ge0$, by  construction we have $\Re v^{t/{n_k}}_{n_k}(z)\ge\Re z$ for all ${z\in\UH}$. It follows that either $v_*\in\Hol(\UH,\UH)$ or $v_*\equiv\infty$. The latter alternative is impossible because, as a holomorphic self-map of $\UH$, $\psi_{t/n_k}$ is non-expansive w.r.t. the hyperbolic distance $\rho$ in~$\UH$, and hence for any $z\in\UH$, we have
\begin{align*}
\rho\big(z,v^{t/{n_k}}_{n_k}(z)\big)=\rho\big(z,\psi_{t/{n_k}}^{\circ n_k}(z)\big)
&\le\sum_{j=1}^{n_k}\rho\big(\psi_{t/{n_k}}^{\circ(j-1)}(z),\psi_{t/{n_k}}^{\circ j}(z)\big) \\
&\le n_k\,\rho(z,z+tf(z)/n_k)\to\frac{t|f(z)|}{2\Re z}
\end{align*}
as $k\to+\infty$. Thanks to~\eqref{EQ_Abel-epsilon} with $\varepsilon:=t/n_k$, we have $h_{t/n_k}\circ v_{n_k}^{t/n_k}=h_{t/n_k}+t$. Passing to the limit as ${k\to+\infty}$, we obtain the desired relation $h\circ v_*=h+t$.

As a Koenigs function of a one-parameter semigroup, $h$ is univalent in $\UH$. Therefore, the equality $h\circ v_*=h\circ v_t^{\phi}$ is only possible if $v_*=v^\phi_t$. Since $v_n^{t/n}\in\Hol(\UH,\UH)$ for any ${n\in\Natural}$, ${t\ge0}$, and since $\Hol(\UH,\UH)$ is a normal family, the latter means that for each ${t\ge0}$, $v_n^{t/n}\to v_t^\phi$ in $\Hol(\UH,\UH)$ as ${n\to+\infty}$. This completes the proof.
 \end{proof}

\begin{proof}[\proofof{Theorem~\ref{TH_BG}}]
Suppose first that $\phi$ is a Bernstein generator. As usual, denote by $(v_t^\phi)$ the one-parameter semigroup in $\BF$ associated with~$\phi$. Since $\phi(z)=\lim_{t\to0^+}\big(z-v_t^\phi(z)\big)/t$ and the limit is attained locally uniformly in~$\UH$, it follows that $\phi$ is real-valued on~$(0,+\infty)$ and that
$$
 (-1)^{n}\phi^{(n)}(x)=\lim_{t\to0^+}\frac{(-1)^{n-1}(v_t^\phi)^{(n)}(x)}t\ge0
$$
for all $x\in(0,+\infty)$ and all~$n\ge2$, because $(-1)^{n-1}\big(v_t^\phi\big)^{(n)}(x)\ge0$ for all $t>0$ and all~$n\in\Natural$.

Since $\phi''(x)\ge0$ for all $x\in(0,+\infty)$, the limit $\phi(0):=\lim_{\Real\ni x\to0^+} \phi(x)$ exists and belongs to $(-\infty,+\infty]$. Considering the equation $\di w/\di t+\phi(w(t))=0$ with the initial condition $w(0)=x$ in a right neighbourhood of the origin on the real line, we see that the solution leaves $\UH$ at finite time if $\phi(0)>0$, which would contradict the fact that $\phi$ is an infinitesimal generator in~$\UH$.\smallskip

Thus, conditions~(i)\,--\,(iii) hold for every Bernstein generator. Suppose now that $\phi$ satisfies (i)\,--\,(iii) and show that in such a case $\phi$ is a Bernstein generator. By (i) and (iii), $\phi'$ is a real-valued non-decreasing function in $(0,+\infty)$. It follows that there exists the limit $\phi'(\infty):=\lim_{\R\ni x\to+\infty}\phi(x)$ belonging to $(-\infty,+\infty]$.

First let us consider the special case $\phi'(\infty)\neq+\infty$. It is worth mentioning that this is done for the sake of clarity rather than because it is logically necessary.
If $\phi'(\infty)\le0$, then $\phi'(x)\le0$ for all $x\in(0,+\infty)$ and hence conditions (i)\,--\,(iii) imply that $f:=-\phi$ is a Bernstein function. Therefore, by Proposition~\ref{PR_Bernstein-func-is-BG}, $\phi$ is a Bernstein generator.

Similarly, if $\phi'(\infty)\in(0,+\infty)$, then $f_1(z):=\phi'(\infty)z-\phi(z)$ is a Bernstein function. Hence $\phi=\phi_1+\phi_2$, where $\phi_1:=-f_1$ and $\phi_2(z):=\phi'(\infty)z$, is a Bernstein generator, because $\phi_1$ and $\phi_2$ belong to~$\Gen(\BF)$ by Proposition~\ref{PR_Bernstein-func-is-BG} and by Lemma~\ref{LM_BG-elementary}, respectively, and because $\Gen(\BF)$ is a convex cone by Corollary~\ref{CR_BernsteinGen-conv-cone}.

Now we present a (bit more sophisticated) proof valid both for finite and infinite values of~$\phi'(\infty)$.
Since $\phi''$ is completely monotone by~(iii), the limit $\phi''(\infty):=\lim_{\R\ni x\to+\infty}\phi''(x)$ exists and satisfies
\begin{equation}\label{EQ_second-derivative-at-infty}
\phi''(x)\ge\phi''(\infty)\ge0\qquad\text{for all~$x>0$}.
\end{equation}
It follows that the function $\varphi(z):=\phi(z)-\phi''(\infty)z^2/2$, $z\in\UH$, satisfies the same conditions~(i)\,--\,(iii) as~$\phi$. In particular, the limit $\varphi'(\infty):=\lim_{\R\ni x\to+\infty}\varphi(x)$ exists and belongs to ${(-\infty,+\infty]}$. In addition, $\varphi''(\infty)=0$. It follows that
\begin{equation}\label{EQ_varphi_prime-over-x}
\lim_{\R\ni x\to+\infty}\frac{\varphi'(x)}{x}=0.
\end{equation}

Using again Corollary~\ref{CR_BernsteinGen-conv-cone} together with the fact that $z\mapsto \varphi''(\infty)z^2$ is Bernstein generator by Lemma~\ref{LM_BG-elementary}, we may conclude that in order to complete the proof it is enough to show that $\varphi\in\Gen(\BF)$.

For $\varepsilon>0$ and $z\in\UH$, we define
\begin{align*}
  \phi_1^\varepsilon(z) & := {\varphi(z) -  \varphi'(1/\varepsilon)z}, \\
  \phi_\varepsilon(z)      & := \phi_1^\varepsilon\big( T_\varepsilon(z)\big)+\varphi'(1/\varepsilon)z,\quad\text{where $~T_\varepsilon(z):=z/(1+\varepsilon z)$}.
\end{align*}
Note that $(-\phi_1^\varepsilon)'$ is completely monotonic in $(0,1/\varepsilon)$. Moreover, $\lim_{\R\ni x\to0^+}\phi_1^\varepsilon(x)=\phi(0)\le0$ and hence $\phi_1^\varepsilon$ is non-positive on~$(0,1/\varepsilon)$. Since $T_\varepsilon$ is a Bernstein function  and $T_\varepsilon\big((0,+\infty)\big)=(0,1/\varepsilon)$, by Remark~\ref{RM_composition} it follows that $-\phi_1^\varepsilon\circ T_\varepsilon$ is also a Bernstein function. Therefore, arguing as above we may conclude that $\phi_\varepsilon\in\Gen(\BF)$ for every  ${\varepsilon>0}$.

Using~\eqref{EQ_varphi_prime-over-x}, we get:
\begin{eqnarray*}
% \nonumber % Remove numbering (before each equation)
  \lim\limits_{\varepsilon\to0^+} \phi_\varepsilon(z) &=&%
       \lim\limits_{\varepsilon\to0^+} \big(\varphi(T_\varepsilon(z))-\varphi'(1/\varepsilon)T_\varepsilon(z)+ \varphi'(1/\varepsilon)z\big)\\
  &=& \lim\limits_{\varepsilon\to0^+} \varphi(T_\varepsilon(z)) ~+~ %
         \lim\limits_{\varepsilon\to0^+} \dfrac{\varepsilon\, \varphi'(1/\varepsilon)\,z^2}{1+\varepsilon z}\\
  &=& \varphi(z) \vphantom{\int^1}.
\end{eqnarray*}
Since by Corollary~\ref{CR_BernsteinGen-conv-cone}, $\Gen(\BF)$ is closed in~$\Hol(\UH,\C)$,  it follows that $\varphi$ is a Bernstein generator and hence the proof is complete.
\end{proof}

\begin{proof}[\proofof{Corollary~\ref{CR_BG-BRFP0}}]
First of all, notice that according to Remarks~\ref{RM_BRFP-semigroups},~\ref{RM_critical-via-radial-limits}, and~\ref{RM_conformal-trans-of-generators}, a one-parameter semigroup $(v_t^\phi)\subset\Hol(\UH,\UH)$ with associated infinitesimal generator~$\phi$ has a BRFP at~$0$ if and only if
\begin{equation}\label{EQ_RFBR-at-0}
\phi(0):=\lim_{x\to0^+}\phi(x)~=~0 \quad\text{and}\quad \phi'(0):=\lim_{x\to0^+}\phi'(x)~\text{~exists finitely}.
\end{equation}

Suppose now that $\phi\in\Gen(\BF)$ and that the above condition~\eqref{EQ_RFBR-at-0} is satisfied. Following the lines of the argument given after the statement of Theorem~\ref{TH_BG}, which was used to deduce Silverstein's formula~\eqref{EQ_Silver-representation}, we represent $\phi''$ as the Laplace transform of a non-negative $\sigma$-finite Borel measure $\mu$ on~$[0,+\infty)$. Since $\phi'$ has a finite limit at~$0^+$, with the help of Tonelli's Theorem, we get
$$
\int_{0}^1\phi''(x)\DI x=\int_{0}^1\Big(\int_0^{+\infty}\!e^{-\lambda x}\,\mu(\di\lambda)\Big)\di x=\mu(\{0\})\,+\!\!\int\limits_{(0,+\infty)}\!\!\!\!\lambda(1-e^{-\lambda})\,\pi(\di\lambda)<+\infty,
$$
where $\pi$ is a Borel measure on~$(0,+\infty)$ given by $\pi(\di\lambda):=\lambda^{-2}\mu|_{(0,+\infty)}(\di\lambda)$. Clearly, the integrability condition we established is equivalent to~\eqref{EQ_BG-BRFP0-int-cond}.

With a similar technique, for all $x>0$ we further obtain
$$
\phi'(x):=\phi'(0)+\mu(\{0\})x\,+\!\!\int\limits_{(0,+\infty)}\!\!\!\!\lambda(1-e^{-\lambda x})\,\pi(\di\lambda).
$$
Extending this formula holomorphically to~$\UH$ and taking into account the first equality in~\eqref{EQ_RFBR-at-0}, we arrive at the desired representation~\eqref{EQ_LeGall-representation}.

Conversely, bearing in mind Bernstein's theorem on completely monotone functions, see e.g. \cite[Theorem~1.4 on p.\,3]{SSV12}, one can easily see that if $\phi$ is given by~\eqref{EQ_LeGall-representation} with some $c\in\Real$, $b\ge0$, and a Borel measure $\pi$ on $(0,+\infty)$ subject to the integrability condition~\eqref{EQ_BG-BRFP0-int-cond}, then $\phi$ is a holomorphic function in~$\UH$ satisfying conditions (i)\,--\,(iii) in Theorem~\ref{TH_BG} and~\eqref{EQ_RFBR-at-0}, with ${\phi'(0)=c}$. Therefore, $\phi$ is a Bernstein generator and the associated one-parameter semigroup $(v_t^{\phi})$ has a BRFP at~$0$.
 Moreover, with the different sign in the definition of the infinitesimal generator taken into account, from \cite[Theorem~1]{CoDiPo06} it follows that $(v_t^{\phi})'(0)=e^{-\phi'(0) t}=e^{-c\,t}$ for all ${t\ge0}$. Finally, the BRFP at~$0$ is the DW-point of $(v_t^{\phi})$ if and only if $(v_t^{\phi})'(0)\le 1$ for all~$t\ge0$, which is equivalent to the inequality ${c\ge0}$.
\end{proof}

\begin{proof}[\proofof{Corollary~\ref{CR_BG-BRFP-infty}}]
With the help of Remarks~\ref{RM_BRFP-semigroups}, \ref{RM_BRFP-radial}, and~\ref{RM_conformal-trans-of-generators} we see that a one-parameter semigroup $(v_t^\phi)\subset\Hol(\UH,\UH)$ with associated infinitesimal generator~$\phi$ has a BRFP at~$\infty$ if and only if
$\phi'(\infty):=\lim_{x\to+\infty}{\phi(x)}/{x}\neq\infty$.
Moreover, if $\phi\in\Gen(\BF)$, then according to Theorem~\ref{TH_BG}, $\phi''(x)\ge0$ for all ${x>0}$ and hence the limit $\lim_{x\to+\infty}\phi'(x)$ exists, finite or infinite, and coincides with~$\phi'(\infty)$.

We omit the rest of the proof since it is very similar to that of Corollary~\ref{CR_BG-BRFP0}.
\end{proof}

\subsection{Conditions for absolute continuity of evolution families}\label{SS_AC-conditions}
Several (necessary and/or sufficient) conditions for a family of holomorphic self-maps satisfying conditions EF1 and EF2 in Definition~\ref{DF_EF} (or conditions REF1 and REF2 in Definition~\ref{DF_REF} respectively) to be an \textit{absolutely continuous} (reverse) evolution family are known, see e.g.\ \cite[Theorem~7.3]{BCM1}, \cite[Section~2]{CDG_LCh}, \cite[Proposition~3.7]{CDG-AnnulusI}, \cite[Section~4.1]{CDG_decr}, \cite[Section~4]{Gum-represent}, \cite[Theorem~1]{GoryainovBRFP}.

The following result improves considerably \cite[Theorem~7.3]{BCM1} for the case of the boundary DW-point (except that we do not keep track on the order of integrability). For a domain $D\in\{\UD,\UH\}$ and a point $\tau\in\partial D$ denote by $\U_\tau$ the set of $\id_D$ and all self-maps $v\in\Hol(D,D)$ for which $\tau$ is the DW-point.
\begin{theorem}\label{TH_diff-one-trajectory}
Let $D\in\{\UD,\UH\}$ and $\tau\in\partial D$. Suppose that $(w_{s,t})_{(s,t)\in\Delta(I)}\subset\U_\tau$ satisfies conditions EF1 and EF2 in Definition~\ref{DF_EF}. If there exits ${z_0\in D}$ such that the map $I\ni t\mapsto w_{0,t}(z_0)$ belongs to $\AC(I)$, then $(w_{s,t})$ is an absolutely continuous evolution family.
\end{theorem}
\begin{proof}
Clearly, we may suppose that $D=\UH$ and that $\tau=\infty$.
For ${t\in I}$ denote
$$
 a(t):=\Re w_{0,t}(z_0)\quad \text{and}\quad b(t):=\Im w_{0,t}(z_0).
$$
By EF2, $w_{0,t_2}(z_0)={w_{t_1,t_2}\big(w_{0,t_1}(z_0)\big)}$.
Hence by Julia's Lemma for $\UH$, see Theorem \ref{TH_Julia-half-plane}, applied to the map $w_{t_1,t_2}$, for any $(t_1,t_2)\in\Delta(I)$ we have
\begin{equation}\label{EQ_w-prime}
1\le w_{t_1,t_2}'(\infty)\le {a(t_2)}/{a(t_1)},
\end{equation}
where the left inequality holds simply because $\infty$ is the DW-point of~$w_{t_1,t_2}$.

\newcommand{\wt}{\hskip-.01em\widetilde{\hskip.01em w}}
The family $(\wt_{s,t})$ defined by $\wt_{s,t}:=L_t\circ w_{s,t}\circ L_s^{-1}$ for all ${(s,t)\in\Delta(I)}$, where $L_t(z):=\big(z-ib(t)\big)/a(t)$, satisfies conditions EF1 and EF2 in Definition~\ref{DF_EF}. Moreover, its elements share the same DW-point at~$1\in \UH$ and a boundary regular fixed point at~$\infty$. By the Chain Rule for angular derivatives, see e.g. \cite[Lemma~2]{CoDiPo06},
$$
\frac{\wt_{0,t_2}'(\infty)}{\wt_{0,t_1}'(\infty)}=\wt'_{t_1,t_2}(\infty)=\frac{a(t_1)}{a(t_2)}\,w'_{t_1,t_2}(\infty)
\quad\text{for any~$~(t_1,t_2)\in\Delta(I)$.}
$$
Using~\eqref{EQ_w-prime} we therefore obtain
$$
 \frac{a(t_1)}{a(t_2)}\le \frac{\wt_{0,t_2}'(\infty)}{\wt_{0,t_1}'(\infty)}\le 1.
$$
Since by the hypothesis $a\colon I\to(0,+\infty)$ is a function of class $\AC$, the latter implies that the map $t\mapsto \wt_{0,t}'(\infty)$ belongs to $\AC(I)$. Recall that the self-maps $\wt_{s,t}$ share the same interior DW-point. Therefore, by \cite[Theorem~4.2]{Gum-represent}, $(\wt_{s,t})$ is an absolutely continuous evolution family. Since by the hypothesis $a,b\in\AC(I)$, it is easy to see that the automorphisms $L_{s,t}:=L_t\circ L_s^{-1}$ also form an absolutely continuous evolution family. Writing $w_{s,t}={L_t^{-1}\circ\wt_{s,t}\circ L_s}$ for all ${(s,t)\in\Delta(I)}$ and applying \cite[Lemma~2.8]{CDG_LCh}, we conclude that $(w_{s,t})$ is also an absolutely continuous evolution family.
\end{proof}

\begin{corollary}
Let $(v_{s,t})_{(s,t)\in\Delta(I)}$ be a topological reverse or usual evolution family in $D\in\{\UD,\UH\}$. Suppose that the following two conditions hold:
\begin{itemize}
\item[(i)] $v_{s,t}(D\cap\Real)\subset D\cap\Real$ for any $(s,t)\in\Delta(I)$;
\item[(ii)] $(v_{s,t})\subset\U_\tau$ for some $\tau\in\overline D$.
\end{itemize}
Then there exists a continuous increasing bijective map $u\colon J\to I$, $J\subset\Real$, such that the (reverse or usual) evolution family $(\tilde v_{s,t})_{(s,t)\in\Delta(J)}$ given by $\tilde v_{s,t}:=v_{u(s),u(t)}$ for all ${(s,t)\in\Delta(J)}$ is absolutely continuous.
\end{corollary}
\begin{proof}
We prove the corollary only for the case $D=\UH$. The proof for ${D=\UD}$ is very similar and therefore, we omit it.

Furthermore, we first suppose that $(v_{s,t})$ is a topological evolution family and that ${\tau\in\partial\UH}$. Then (i) implies that $\tau\in\{0,\infty\}$. Using if necessary the change of variable ${z\mapsto1/z}$, we may suppose that $\tau=\infty$. Choose any $z_0\in(0,+\infty)$. The map $$I\ni t\mapsto f(t):=t+v_{0,t}(z_0)-z_0$$ is strictly increasing and continuous (because combining $v_{s,t}'(\infty)\ge1$ and Julia's lemma (Theorem \ref{TH_Julia-half-plane}) implies $\Re v_{s,t}(z) \ge \Re z$ on $\UH$). Therefore, it has a strictly increasing continuous inverse $u\colon J\to I$ mapping some interval $J\subset\Real$ onto~$I$. Clearly, the functions $\tilde v_{s,t}:=v_{u(s),u(t)}$,  ${(s,t)\in\Delta(J)}$, form a topological evolution family. Moreover, by construction,
\begin{equation}\label{EQ_v-reparam}
 \tilde v_{0,t}(z_0)=v_{0,u(t)}(z_0)=-u(t)+f(u(t))+z_0=-u(t)+t+z_0\in\Real
\end{equation}
for all ${t\in J}$. If $t_2\ge t_1$, $t_1,t_2\in I$, then  $v_{0,t_2}(z_0)=v_{t_1,t_2}(v_{0,t_1}(z_0))\ge v_{0,t_1}(z_0)$ and hence $f(t_2)-f(t_1)\ge t_2-t_1$. It follows that $u$ is Lipschitz continuous. Using~\eqref{EQ_v-reparam}  and Theorem~\ref{TH_diff-one-trajectory}, we conclude that the evolution family $(\tilde v_{s,t})$ is absolutely continuous.

If now $(v_{s,t})$ is a topological evolution family with the DW-point ${\tau\in\UH}$, then ${\tau\in(0,+\infty)}$.  Fix some $z_0\in(0,\tau)$ and define the family $(\tilde v_{s,t})$ literally by the same formula as above. In order to make sure that $f$ is strictly increasing in~$I$ also in this case, notice that since $t\mapsto v_{0,t}(z_0)\in(0,+\infty)$ is continuous and since elements of a topological evolution family are univalent \cite[Proposition~2.4]{CD_top-LTh}, we have $v_{0,t}\big((0,\tau)\big)\subset(0,\tau)$ for all ${t\in I}$. Hence the monotonicity of~$f$ follows from the fact that the holomorphic maps $v_{s,t}$ do not increase the hyperbolic distance in~$\UH$. For $z\in\{z_0,\tau\}$, the map $t\mapsto \tilde v_{0,t}(z)$ belongs $\AC(J)$.  Again by univalence of elements of a topological evolution family,  $\tilde v_{0,t}(z_0)\neq \tilde v_{0,t}(\tau)$ for all ${t\in J}$. By \cite[Proposition~3.7]{CDG-AnnulusI}, it follows that $(\tilde v_{s,t})$ is an absolutely continuous evolution family.

Suppose now that $(v_{s,t})$ is a topological \textit{reverse} evolution family. If $I$ is of the form $[0,T]$, $0<T<+\infty$, then applying the above argument to the topological evolution family $(\hat v_{s,t})_{(s,t)\in\Delta(I)}$ defined by $\hat v_{s,t}:=v_{T-t,T-s}$, we directly arrive to the desired conclusion.

To cover the case $I=[0,T)$, $0<T\le+\infty$, the argument has to be slightly modified. As above, we suppose that $\tau=\infty$ or $\tau\in(0,+\infty)$ and fix some $z_0\in(0,\tau)$. Next we choose an increasing sequence $(T_n)\subset (0,T)$ converging to~$T$ and set
$$
f(s):=s+\sum_{n=1}^{+\infty}{\vphantom{\sum}}'\,\frac{z_0-v_{s,T_n}(z_0)}{a_n}\ind_{[0,T_n]}(s),\quad a_n:=2^n\max_{s\in[0,T_n]}\big|z_0-v_{s,T_n}(z_0)\big|,
$$
where the prime in the summation symbol means that, in order to cover the possibility of $v_{0,T_n}=\id_\UH$ for some $n\in\Natural$, we omit all terms for which $a_n=0$.

Let $u:=f^{-1}\colon J\to I$, where $J:=f(I)$. Fix some $n\in\Natural$ and notice that $f(s)=f_n(s)+a_n^{-1}\big(z_0-v_{s,T_n}(z_0)\big)$, ${s\in [0,T_n]}$, where $f_n$ is a strictly increasing function on~$I$. Since $s\mapsto {z_0-v_{s,T_n}(z_0)}$ is non-decreasing, the function $f_n\circ u$ is Lipschitz continuous on~$f([0,T_n])$. This fact allows us to show, arguing to as above, that $\big(v_{T_n-u(t),T_n-u(s)}\big)$ is an absolutely continuous reverse evolution family on~$[0,T_n]$.
Since ${n\in\Natural}$ is arbitrary in this argument, we may conclude that $(v_{u(s),u(t)})$ is an absolutely continuous evolution family, as desired.
\end{proof}

Now we prove an analogue of Theorem~\ref{TH_diff-one-trajectory} for reverse evolution families.

\begin{theorem}\label{TH_reverse-diff-one-trajectory}
Let $D\in\{\UD,\UH\}$ and $\tau\in\partial D$. Let $(v_{s,t})_{(s,t)\in\Delta(I)}\subset\U_\tau$ be a topological reverse evolution family. If there exists ${z_0\in D}$ such that the map $I\ni t\mapsto v_{0,t}(z_0)$ belongs to $\AC(I)$, then the reverse evolution family $(v_{s,t})$ is absolutely continuous.
\end{theorem}
\begin{remark}
 In contrast to Theorem~\ref{TH_diff-one-trajectory}, the hypothesis in the above theorem contains the requirement that $(v_{s,t})$ is continuous in~$(s,t)$. This assumption is essential for the proof we give below. However, we do not know whether Theorem~\ref{TH_reverse-diff-one-trajectory} remains valid if we assume that $(v_{s,t})$ satisfies REF1 and REF2 from Definition~\ref{DF_REF} without requiring \textit{a priori} any regularity in $s$ and~$t$.
\end{remark}

To prove Theorem~\ref{TH_reverse-diff-one-trajectory} we need the following two lemmas.

\begin{lemma}\label{LM_compare-with-id}
Let $f\in\Hol(\UH,\UH)$, $f\neq\id_\UH$, and $z_0\in\UH$. Suppose  $\infty$ is the Denjoy\,--\,Wolff point of~$f$. Then for any compact set $K\subset\UH$ there exists a constant $A_K>0$  such that
\begin{equation}\label{EQ_compare-with-id}
\max_{z\in K}|f(z)-z|\le A_K |f(z_0)-z_0|.
\end{equation}
The constant $A_K$ depends on~$K$ and $z_0$, but not on~$f$.
\end{lemma}
\begin{proof}
By Julia's Lemma for $\UH$, see Theorem \ref{TH_Julia-half-plane}, either $g(z):=f(z)-z$ is constant, or $g\in\Hol(\UH,\UH)$. In the former case, \eqref{EQ_compare-with-id}~holds with $A_K:=1$. In the latter case, it is enough to apply Harnack's inequality to the function~$g$ in the same manner as in the proof of Proposition~\ref{PR-HVF-intergrability}.
\end{proof}

\begin{lemma}\label{LM_Lipschitz}
Under the hypothesis of Theorem~\ref{TH_reverse-diff-one-trajectory}, the map
$
 {(s,t,\zeta)\mapsto v_{s,t}^{-1}(\zeta)}
$
is Lipschitz in~$\zeta$ on each compact subset~$K$ of
\begin{equation*}%\label{EQ_the-set-Q}
 Q:=\{(s,t,\zeta)\colon (s,t)\in\Delta(I),~\zeta\in v_{s,t}(D)\},
\end{equation*}
i.e., for any $(s,t)\in\Delta(I)$ and $\zeta_1,\zeta_2\in v_{s,t}(D)$ such that ${(s,t,\zeta_j)\in K}$, ${j=1,2}$, we have
\begin{equation}\label{EQ_Lipschitz}
|v_{s,t}^{-1}(\zeta_2)-v_{s,t}^{-1}(\zeta_1)|\le C_K |\zeta_2-\zeta_1|,
\end{equation}
where $C_K>0$ is a constant depending on~$K$.
\end{lemma}
\begin{proof}
By \cite[Theorem 3.16]{FHS-Diss} or combining Remark~\ref{RM_REF-EF} and \cite[Proposition~2.4]{CD_top-LTh}, we see that for any $(s,t)\in\Delta(I)$  the function $v_{s,t}$ is univalent in~$D$. From Carath\'eodory's Kernel Convergence Theorem, see e.g. \cite[\S{}II.5]{Goluzin} or \cite[Sect.\,3.5]{BCD-Book}, see also \cite[Sect.\,2]{AngularExt}, it follows that if ${(s_0,t_0,\zeta_0)\in Q}$, then for any $r>0$ such that $\overline{\,\UD_r(\zeta_0)}\subset v_{s_0,t_0}(D)$ there exists ${\delta>0}$ such that $\overline{\,\UD_r(\zeta_0)}\subset v_{s,t}(D)$ for any ${(s,t)\in\Delta(I)}$ with ${|s-s_0|<\delta}$ and ${|t-t_0|<\delta}$.
This implies that $Q$ is relatively open in ${\Delta(I)\times\C}$.

For $r>0$ and $(s_0,t_0,\zeta_0)\in\Delta(I)\times\C$, denote
$$
 B_r(s_0,t_0,\zeta_0):=\big\{(s,t,\zeta)\in\Delta(I)\times\C\colon |s-s_0|<r,\,|t-t_0|<r,\,|\zeta-\zeta_0|<r\big\}.
$$
Let $K$ be a compact subset of~$Q$. Then there exists  $\varepsilon>0$ such that $K':=\overline{\bigcup_{(s,t,\zeta)\in K}B_{\varepsilon}(s,t,\zeta)}$ is also a compact subset of~$Q$.

Again by Carath\'eodory's Kernel Convergence Theorem, the map $Q\ni(s,t,\zeta)\mapsto v_{s,t}^{-1}(\zeta)$ is continuous.  Therefore, it is sufficient to establish~\eqref{EQ_Lipschitz} under the additional condition that $|\zeta_2-\zeta_1|\le\varepsilon$. To this end notice that, by continuity, the function $Q\ni(s,t,\zeta)\mapsto (v_{s,t}^{-1})'(\zeta)$ is bounded on~$K'$, i.e. $|(v_{s,t}^{-1})'(\zeta)|\le C_K$ for all $(s,t,\zeta)\in K'$ and some constant ${C_K>0}$.

Therefore, if $(s,t,\zeta_1)\in K$, then $\{(s,t,\zeta)\colon|\zeta-\zeta_1|\le\varepsilon\}\subset K'$ and hence
$$
 \big|v_{s,t}^{-1}(\zeta_2)-v_{s,t}^{-1}(\zeta_1)\big|\le C_K|\zeta_2-\zeta_1|
$$
whenever $|\zeta_2-\zeta_1|\le\varepsilon$.
\end{proof}

\begin{proof}[\proofof{Theorem~\ref{TH_reverse-diff-one-trajectory}}]
First of all, obviously, we may assume that ${D=\UH}$ and ${\tau=\infty}$. Further, fix some $\Tau\in I$. According to the definition of a reverse evolution family, it is sufficient to show that for any ${z\in\UH}$ there exists a constant $M_{\Tau,z}>0$ such that for any $s$, $t_1$, $t_2$ satisfying $0\le s\le t_1\le t_2\le \Tau$, we have
\begin{equation*}%\label{EQ_REV-OT-we-need}
\big|v_{s,t_2}(z)-v_{s,t_1}(z)\big|\le M_{\Tau,z}\big|v_{0,t_2}(z_0)-v_{0,t_1}(z_0)\big|.
\end{equation*}

Apply Lemma~\ref{LM_Lipschitz} with
$$
 K=K_\Tau:=\big\{\big(0,t_1,v_{0,t_2}(z_0)\big)\colon  0\le t_1\le t_2\le \Tau \big\}.
$$
Clearly, $K_\Tau$ is compact. Moreover,  $v_{0,t_2}(\UH)\subset v_{0,t_1}(\UH)$ for any ${t_2\ge t_1}$, ${t_1,t_2\in I}$, because $v_{0,t_2}=v_{0,t_1}\circ v_{t_1,t_2}$; hence, $K_\Tau\subset Q$. In view of condition~REF2 in Definition~\ref{DF_REF}, inequality~\eqref{EQ_Lipschitz} for  ${\zeta_j:=v_{0,t_j}(z_0)}$, $j=1,2$, $s:=0$, and ${t:=t_1}$, yields
\begin{equation*}%\label{EQ_REV-OT-we-have-1}
\big|v_{t_1,t_2}(z_0)-z_0\big|\le C_{K_\Tau} \big|v_{0,t_2}(z_0)-v_{0,t_1}(z_0)\big|\quad \text{for any $(t_1,t_2)\in\Delta([0,\Tau])$.}
\end{equation*}
Now fix an arbitrary $z\in\UH$. Using Lemma~\ref{LM_compare-with-id} and the above inequality, we obtain
\begin{equation}\label{EQ_REV-OT-we-have-2}
\big|v_{t_1,t_2}(z)-z\big|\le A_{\{z\}}C_{K_\Tau} \big|v_{0,t_2}(z_0)-v_{0,t_1}(z_0)\big|\quad \text{for any $(t_1,t_2)\in\Delta([0,\Tau])$.}
\end{equation}

Since the hyperbolic and Euclidean distances are equivalent on the compact set
$
{\{v_{s,t}(z)\colon (s,t)\in\Delta([0,\Tau])\}}
$
and since the holomorphic mappings~$v_{s,t_1}\colon \UH\to\UH$, $s\in[0,t_1]$, do not increase the hyperbolic distance, from inequality~\eqref{EQ_REV-OT-we-have-2} it follows that for a suitable constant ${M_{\Tau,z}>0}$ and for any $s$, $t_1$, $t_2$ satisfying $0\le s\le t_1\le t_2\le \Tau$, we have
$$
\big|v_{s,t_2}(z)-v_{s,t_1}(z)\big|=\big|v_{s,t_1}(v_{t_1,t_2}(z))-v_{s,t_1}(z)\big|\le M_{\Tau,z} \big|v_{0,t_2}(z_0)-v_{0,t_1}(z_0)\big|.
$$
This completes the proof.
\end{proof}

Next we turn to families of holomorphic self-maps of $\UH$ formed by (extensions) of Bernstein functions. First we establish a (quantitative) rigidity type result. In contrast to arbitrary holomorphic self-maps, for which the Burns\,--\,Krantz rigidity property involves angular derivatives up to the third order, see e.g. \cite{Shoikhet-rigidity} and \cite{Dubinin}, for functions in~$\BF$, existence of finite limits of the first \textit{two}  derivatives at $z=0$ is sufficient. Recall that by definition of a Bernstein function, if $f\in\BF$ then $f'$ as well as all higher order derivatives of~$f$ are real-valued and monotonic on ${(0,+\infty)}$. Hence there exist (finite or infinite) limits
$$
 f^{(n)}(0):=\lim_{\R\ni x\to 0^+}f^{(n)}(x),\quad n\in\N.
$$
Moreover, there exists a non-negative finite limit $f(0):=\lim_{\R\ni x\to 0^+}f(x)$.

\begin{proposition}[Rigidity property of Bernstein functions]\label{PR_Bernstein-rigidity}
Let $f\in\BF$ and suppose that $f'(0)$ and $f''(0)$ are finite. Then
\begin{equation}\label{EQ_Bernstein-rigidity}
|f(z)-z|\le f(0)+\big|(f'(0)-1)z\big|+\tfrac12\big|f''(0)z^2\big|
\end{equation}
for any $z\in\UH$.
\end{proposition}
\begin{proof}
Let $f\in\BF$. Then for all $z\in\UH$,
\begin{equation}\label{EQ_RG-repres}
 f(z)=\alpha+\beta z+\int\limits_0^{+\infty}(1-e^{-\lambda z})\,\tau(\di\lambda),
\end{equation}
where $\alpha,\beta\ge0$ and $\tau$ is a Borel non-negative measure~$\tau$ on~$(0,+\infty)$ with $\int_0^\infty \min\{\lambda,1\}\,\tau(\di\lambda)<\infty$.
With the help of Levi's Monotone Convergence Theorem, it is elementary to show that
\begin{subequations}\label{EQ_RG-repres-A}
\begin{align}
               f(0)   &= \alpha,\\
[0,+\infty]\ni f'(0)  &= \beta\,+\,\int_0^{+\infty}\!\lambda\,\tau(\di\lambda),\\
[-\infty,0]\ni f''(0) &= -\,\int_0^{+\infty}\!\lambda^2\,\tau(\di\lambda).
\end{align}
\end{subequations}

Fix some $z\in\UH$.
Note that
\begin{equation}\label{EQ_elementary}
|\lambda z-(1-e^{-\lambda z})|\le\frac{|e^{-\lambda z}|}{2}|\lambda z|^2 <\frac{|\lambda z|^2}{2}\quad\text{for all~}~\lambda>0.
\end{equation}
To deduce this elementary inequality, it is sufficient to fix an arbitrary $\theta\in\Real$ and apply Cauchy's Mean Value Theorem twice starting with the pair of functions $f_1(t):=\Re\big(e^{-i\theta}z^{-2}\big(tz-(1-e^{-tz})\big)\big)$ and $f_2(t):=t^2$,  $t\in[0,\lambda]$.

Integrating inequality~\eqref{EQ_elementary} and combining it with~\eqref{EQ_RG-repres} and~\eqref{EQ_RG-repres-A}, we obtain
\begin{align*}
\frac{\big|f''(0)z^2\big|}2 ~=~ \displaystyle \int_0^{+\infty}\!\!\frac{\,|\lambda z|^2}2\,\tau(\di\lambda)
                         & \ge~ \displaystyle \left| \int_0^{+\infty}\!\!(1-e^{-\lambda z})\,\tau(\di\lambda)
                                                 - \int_0^{+\infty}\!\lambda z\,\tau(\di\lambda)         \right| \\[1.5ex]
                         &= \big|f(z)-z-f(0)-(f'(0)-1)z\big|.
\end{align*}
Thanks to the triangle inequality, this immediately implies~\eqref{EQ_Bernstein-rigidity}.
\end{proof}

\begin{remark}\label{RM_two-ways-to-define-derivatives}
Note that by Theorem~\ref{TH_properties-of-BF}\,(B), the unrestricted (and hence also angular) limit of $f\in\BF$ at~$0$ exists and equals $f(0):=\lim_{x\to0^+}f(x)$. Moreover, using formulas \eqref{EQ_RG-repres} and~\eqref{EQ_RG-repres-A}, together the integrability condition $\int_0^{\infty}\min\{\lambda,1\}\,\tau(\di\lambda)<+\infty$ for the measure~$\tau$ in~\eqref{EQ_RG-repres}, it is not difficult to see that  if $f\in\BF$ and the limit
$$
f''(0):=\lim_{x\to0^+}f''(x)
$$
is finite, then the unrestricted limits
$$
\lim_{\UH\ni z\to 0}\frac{f(z)-f(0)}{z-0}\quad\text{~and~}\quad \lim_{\UH\ni z\to 0}\frac{f'(z)-f'(0)}{z-0}
$$
exist, are finite and equal to~$f'(0)$ and $f''(0)$, respectively. Furthermore, thanks to the monotonicity of $f'$ and $f''$ on $(0,+\infty)$, a sort of converse is also true: for $f\in\BF$, existence of finite angular derivatives of the first and second order at~$0$ implies that $f'(0)$ and $f''(0)$ are finite. In particular, the two ways to define $f'(0)$ and $f''(0)$, as angular derivatives and via limits as $\Real\ni x\to0^+$, agree.
\end{remark}

Now we will use the rigidity property proved above in order to give sufficient conditions for a (reverse) evolution family of Bernstein functions to be absolutely continuous.

\begin{theorem}\label{TH_AC-BF-EF}
Let $(w_{s,t})_{(s,t)\in\Delta(I)}$ be a family in~$\BF$ satisfying conditions EF1 and EF2 in Definition~\ref{DF_EF}.  Suppose that the following assertions hold:
\begin{itemize}
\item[(i)] $w_{0,t}(0)=0$ for all~$t\in I$;
\item[(ii)] $w_{0,t}'(0)$ and $w_{0,t}''(0)$ are finite for all~$t\in I$;
\item[(iii)] the mappings $I\ni t\mapsto w'_{0,t}(0)$ and $I\ni t\mapsto w''_{0,t}(0)$ are of class~$\AC(I)$.
\end{itemize}
Then $(w_{s,t})$ is an absolutely continuous evolution family.
\end{theorem}
Before proving the theorem, we obtain a corollary stating that analogous assertion holds for \textit{reverse} evolution families.
\begin{corollary}\label{CR_AC-BF-REF}
Let $(v_{s,t})_{(s,t)\in\Delta(I)}$ be a family in~$\BF$ satisfying conditions REF1 and REF2 in Definition~\ref{DF_REF}.  Suppose that the following assertions hold:
\begin{itemize}
\item[(i)] $v_{0,t}(0)=0$ for all~$t\in I$;
\item[(ii)] $v_{0,t}'(0)$ and $v_{0,t}''(0)$ are finite for all~$t\in I$;
\item[(iii)] the mappings $I\ni t\mapsto v'_{0,t}(0)$ and $I\ni t\mapsto v''_{0,t}(0)$ are of class~$\AC(I)$.
\end{itemize}
Then $(v_{s,t})$ is an absolutely continuous reverse evolution family.
\end{corollary}
\begin{proof}
According to Remark~\ref{RM_REF-EF}, it is sufficient to show that for any $\Tau\in I$, the family $(w_{s,t})_{0\le s\le t\le \Tau}$ defined by $w_{s,t}:=v_{\Tau-t,\Tau-s}$ is an evolution family over~${[0,\Tau]}$.

In the same way as in the proof of Theorem~\ref{TH_AC-BF-EF} given below, one can ensure that $v_{0,s}'(0)$ does not vanish for any~${s\in I}$. Differentiating the relation $v_{0,\Tau-t}\circ w_{0,t}=v_{0,\Tau}$, we find that
$$
  w'_{0,t}(0)=\frac{v'_{0,\Tau}(0)}{v'_{0,\Tau-t}(0)},\quad%
  w''_{0,t}(0)=\frac{1}{v'_{0,\Tau-t}(0)}%
                \Big(v''_{0,\Tau}(0)-v''_{0,\Tau-t}(0)\Big[\frac{v'_{0,\Tau}(0)}{v'_{0,\Tau-t}(0)}\Big]^{\!2\,}\Big).
$$

It follows that $(w_{s,t})$ satisfies the hypothesis of Theorem~\ref{TH_AC-BF-EF} with $I$ replaced by ${[0,\Tau]}$. Thus, it is an absolutely continuous evolution family, as desired.
\end{proof}
\begin{proof}[\proofof{Theorem~\ref{TH_AC-BF-EF}}]
By the hypothesis, the elements of the evolution family $(w_{s,t})$ are holomorphic self-maps of $\UH$ with a boundary fixed point at~$z=0$, see Remark~\ref{RM_two-ways-to-define-derivatives} and Section~\ref{sec:BRFP}.
In particular, $w'_{0,s}(0)\neq0$ for any $s\in I$.
Therefore, using condition~EF2, for any $(s,t)\in\Delta(I)$ we obtain
\begin{equation}\label{EQ_ChainRule-EF}
 w_{s,t}'(0)=\frac{w'_{0,t}(0)}{w'_{0,s}(0)},\quad w''_{s,t}(0)=\frac{1}{\big(w'_{0,s}(0)\big)^2}\Big(w''_{0,t}(0)-w''_{0,s}(0)\frac{w'_{0,t}(0)}{w'_{0,s}(0)}\Big).
\end{equation}

Let us first show that $w_{0,t}$ is univalent for any~$t\in I$. To this end we adapt the proof of \cite[Proposition~2.4]{CD_top-LTh}. Fix ${z_1,z_2\in\UH}$, ${z_1\neq z_2}$. Denote
$$
 J:=\big\{s\in I\colon w_{0,s}(z_1)\neq w_{0,s}(z_2)\big\}.
$$
Clearly, $0\in J$. Moreover, $J$ is an interval, and it is relatively open in~$I$. Indeed, fix any ${s\in J}$. If ${0\le s'\le s}$, then condition~EF2 implies that ${s'\in J}$. Furthermore, using~(iii) and~\eqref{EQ_ChainRule-EF}, we see that ${w_{s,t}'(0)\to 1}$ and ${w_{s,t}''(0)\to 0}$ as ${t\to s^+}$. Applying Proposition~\ref{PR_Bernstein-rigidity} with ${f:=w_{s,t}}$, we conclude that there exists $\varepsilon>0$, which may depend on~$s$, such that if $(s,t)\in\Delta(I)$ and $t<s+\varepsilon$, then
\begin{multline*}
 \big|w_{0,t}(z_2)-w_{0,t}(z_1)-\big(w_{0,s}(z_2)-w_{0,s}(z_1)\big)\big| \le\\
 \big|w_{s,t}(w_{0,s}(z_2))-w_{0,s}(z_2)\big|+\big|w_{s,t}(w_{0,s}(z_1))-w_{0,s}(z_1)\big|
                                               <|w_{0,s}(z_2)-w_{0,s}(z_1)|
\end{multline*}
and hence $t\in J$.

Suppose that $J\neq I$. Then there exists $t_0\in I\setminus\{0\}$ such that $\zeta_0:=w_{0,t_0}(z_1)= w_{0,t_0}(z_2)$ while for any ${t\in[0,t_0)}$,  $w_{0,t}(z_1)\neq w_{0,t}(z_2)$. Using again (iii) and~\eqref{EQ_ChainRule-EF} we see that there exists $s_0\in[0,t_0)$ such that
$$
 a:=\sup_{s_0\le t\le t_0}\Big(\big|w'_{s_0,t}(0)-1\big|+\frac{|w''_{s_0,t}(0)|^2}{2}\Big)<1.
$$
Therefore, by Proposition~\ref{PR_Bernstein-rigidity} applied to $z:=1$ and $f:=w_{s_0,t}$,  for all ${t\in[s_0,t_0]}$ we have ${w_{s_0,t}(1)\in\!\overline{\,\UD_{a}(1)}}\subset\UH$. Since the holomorphic self-maps $w_{s_0,t}$ are non-expansive w.r.t. the hyperbolic distance in~$\UH$, it follows that
$$
 \{w_{0,t}(z_j)\colon j=1,2,~s_0\le t\le t_0\}=\{w_{s_0,t}(w_{0,s_0}(z_j))\colon j=1,2,~s_0\le t\le t_0\}
 \subset K,
$$
where $K\subset\UH$ is a closed ball w.r.t. the hyperbolic distance in~$\UH$. Trivially, $K$ is compact.
Again combining (iii), \eqref{EQ_ChainRule-EF}, and Proposition~\ref{PR_Bernstein-rigidity}, we see that ${w_{t,t_0}(z)\to z}$ uniformly on~$K$ as $[s_0,t_0]\ni t\to t_0$. Hence
\begin{equation}\label{EQ_zj-to-zete0}
  \big|\zeta_0-w_{0,t}(z_j)\big|=\big|w_{t,t_0}(w_{0,t}(z_j))-w_{0,t}(z_j)\big|\to0\quad\text{as $~[s_0,t_0]\ni t\to t_0$}.
\end{equation}

Fix some neighbourhood~$U$ of $\zeta_0$ contained in~$\UH$ along with its boundary. By~\eqref{EQ_zj-to-zete0}, for all $t<t_0$ close enough to~$t_0$, the function $f_t(z):=w_{t,t_0}(z)-\zeta_0$ has two distinct zeros in~$U$, namely, $z=w_{0,t}(z_j)$, $j=1,2$. At the same time, $f_t(z)\to z-\zeta_0$ uniformly in~$U$ as ${[s_0,t_0]\ni t\to t_0}$. The limit function has only one simple zero in~$U$, which contradicts Hurwitz's Theorem. Therefore, $J=I$, i.e. ${w_{0,t}(z_1)\neq w_{0,t}(z_2)}$ for all~${t\in I}$.

Arguing as above we see that for any $\Tau\in I$, there exists a compact set $K_\Tau\subset\UH$ such that $w_{0,t}(z_j)\in K_\Tau$, $j=1,2$, for any~$t\in[0,\Tau]$. Hence in view of (iii),~\eqref{EQ_ChainRule-EF}, and Proposition~\ref{PR_Bernstein-rigidity}, the functions $I\ni t\mapsto w_{0,t}(z_j)$, $j=1,2$, are of class $\AC(I)$. Thus, the conclusion of the theorem follows from~\cite[Proposition~3.7]{CDG-AnnulusI}.
\end{proof}

\subsection{The second angular derivative of evolution families in $\BF$}\label{SS_2nd-derovative}
An elementary result from the theory of ODE is that for the solution $w(t)=w_{s,t}(z)$ to an initial value problem of the form
$$
\frac{\di w}{\di t}+\phi(w,t)=0,\quad w(s)=z,
$$
where $\phi(\cdot,t)$ is holomorphic in the neighbourhood of~$w=0$ and vanishes at that point, we have
\begin{align}
\label{EQ_1order-der}
 w'_{s,t}(0)&=\exp\Big(-\int_s^t\phi'(0,\xi)\,\di \xi\Big),\\
\label{EQ_2order-der}
 w''_{s,t}(0)&=-\,w_{s,t}'(0)\int_s^t\phi''(0,\xi)w'_{s,\xi}(0)\,\di \xi.
\end{align}

If now $\phi$ is Herglotz vector field in $\UH$ and the associated evolution family $(w_{s,t})$ has a BRFP at~$\sigma=0$, then formula~\eqref{EQ_1order-der}, where $w'_{s,t}(0)$ and $\phi'(0,t)$ are to be understood as angular derivatives, is justified by Theorem~\ref{thm:boundary_derivative}.

In view of Theorem~\ref{TH_AC-BF-EF}, it is worth extending \eqref{EQ_2order-der} to evolution families of Bernstein functions.
\begin{proposition}
Let $(w_{s,t})_{(s,t)\in\Delta(I)}$ be an absolutely continuous evolution family contained in $\BF$ and let $\phi$ be the associated Herglotz vector field. Suppose that for any ${(s,t)\in\Delta(I)}$, ${w_{s,t}(x)\to0}$ as $\Real\ni x\to0^+$. Then formula~\eqref{EQ_2order-der} holds, with $w'_{s,\xi}(0)$, $w''_{s,t}(0)$, and $\phi''(0,\xi)$ understood as the limits of $w'_{s,\xi}(x)$, $w''_{s,t}(x)$, and of $\phi''(x,\xi)$ as $\Real\ni x\to0^+$, respectively. These limits as well as the integrals in the r.h.s.\ of~\eqref{EQ_1order-der} and \eqref{EQ_2order-der} can be infinite.
\end{proposition}
\begin{proof}
Fix $(s,t)\in\Delta(I)$. An elementary computation shows that for any $x>0$,
\begin{equation}\label{EQ_log-der-of-der}
\frac{w''_{s,t}(x)}{w'_{s,t}(x)}=\,-\int_{s}^t\phi''\big(w_{s,\xi}(x),\xi\big)w'_{s,\xi}(x)\,\di\xi.
\end{equation}
By the hypothesis, the mappings $w_{s,\xi}$ belong to~$\BF$. In particular, for each $\xi\in[s,t]$, $x\mapsto w_{s,\xi}(x)$ is positive and non-decreasing on $(0,+\infty)$. Moreover, by Corollary~\ref{CR_Herglotz-VF-for-Bernstein-evol-fam}, $\phi(\cdot,\xi)$ is a Bernstein generator for a.e. $\xi\in[s,t]$. Therefore, using Theorem~\ref{TH_BG} we may conclude that for a.e.\ $\xi\in[s,t]$ fixed, $x\mapsto \phi''\big(w_{s,\xi}(x),\xi\big)$ is a non-negative non-increasing function. Finally, for each $\xi\in[s,t]$, $(0,+\infty)\ni x\mapsto w'_{s,\xi}(x)$ is positive and non-increasing, because $w_{s,\xi}$ is a Bernstein function and because it is univalent as an element of absolutely continuous evolution family. As a result, one can apply the Monotone Convergence Theorem to pass to the limit in~\eqref{EQ_log-der-of-der}:
\begin{equation}\label{EQ_LIM-in-log-der-of-der}
\lim_{x\to0^+}\frac{w''_{s,t}(x)}{w'_{s,t}(x)}
=\,-
\int_{s}^t\phi''(0,\xi)w'_{s,\xi}(0)\,\di\xi.
\end{equation}
In order to justify that we can now deduce~\eqref{EQ_2order-der} by multiplying both sides of~\eqref{EQ_LIM-in-log-der-of-der} by~$w_{s,t}'(0)$, it remains to notice that the indefinite form $0\cdot\infty$ cannot occur in the r.h.s.\ of~\eqref{EQ_2order-der}. Indeed, on the one hand, as we have seen in the proof of Theorem~\ref{TH_AC-BF-EF}, $w'_{s,t}(0)\neq0$. On the other hand, if $w'_{s,t}(0)=+\infty$, then $\log w'_{s,t}(x)\to+\infty$ as ${x\to0^+}$ and hence the limit in \eqref{EQ_LIM-in-log-der-of-der} must be infinite,
%{\color{magenta} I could not understand the last reasoning starting from ``and hence ...''. Do we need some properties of Bernstein functions, maybe?)}
 which again excludes the possibility for an indefinite form.
\end{proof}

\end{document}